\documentclass[11pt]{article}
\usepackage[a4paper]{geometry}
\usepackage[latin1]{inputenc}
\usepackage{amsmath}
\usepackage{amssymb}
\usepackage{amsthm}
\usepackage{enumerate}
\usepackage{graphicx}
\usepackage{color}
\usepackage{cite}
\usepackage{dsfont} 

\numberwithin{equation}{section}

\newcommand{\N}{\mathbb{N}}
\newcommand{\Z}{\mathbb{Z}}
\newcommand{\Q}{\mathbb{Q}}

\newcommand{\F}{\mathcal{F}}
\newcommand{\Ch}{\mathcal{C}}
\newcommand{\R}{\mathbb{R}}
\newcommand{\PR}{\mathbb{P}}
\newcommand{\ESP}{\mathbb{E}}

\newcommand{\sas}{\mathcal{S}\alpha\mathcal{S}}

\newcommand{\Za}[1]{\mathrm{Z}_{\alpha} \left( {#1} \right) }
\newcommand{\LL}{\mathrm{L}^2}
\newcommand{\Lun}{\mathrm{L}^1}

\newcommand{\La}{\mathrm{L}^{\alpha}}
\newcommand{\Li}{\mathrm{L}^{\infty}}

\newtheorem{Theo}{Theorem}[section]
\newtheorem{Lem}{Lemma}[section]
\newtheorem{Prop}{Proposition}[section]
\newtheorem{Rem}{Remark}[section]

\def\o{\omega}
\def\O{\Omega}

\def\al{\alpha}
\def\ga{\gamma}

\newcommand{\supp}{\mathrm{supp}}

\def\sign{\mathrm{sgn }}
\def\al{\alpha}

\def\tcd{Dominated Convergence Theorem}
\def\cov{\mathrm{cov}}
\def\var{\mathrm{Var}}
\def\nor{\mathrm{nor}}

\title{Linear Multifractional Stable Motion: wavelet estimation of $H(\cdot)$ and $\al$ parameters}

\author{Antoine Ayache \\UMR CNRS 8524, Laboratoire Paul Painlev\'e, B\^at. M2\\
  Universit\'e Lille 1\\ 59655 Villeneuve d'Ascq Cedex, France\\
E-mail: \texttt{Antoine.Ayache@math.univ-lille1.fr}\\
\ 
\and
 Julien Hamonier \footnote{Corresponding author}\\
  FR CNRS 2956,  LAMAV, \\
  Institut des Sciences et Techniques de Valenciennes,\\
  Universit\'e de Valenciennes et du Hainaut Cambr\'esis, \\
  F-59313 - Valenciennes Cedex 9, France\\
E-mail: \texttt{Julien.Hamonier@univ-valenciennes.fr}
}

\begin{document}

\maketitle

\begin{abstract}
Linear Fractional Stable Motion (LFSM) of Hurst parameter $H$ and of stability parameter $\al$, is one of the most classical extensions of the well-known Gaussian Fractional Brownian Motion (FBM), to the setting of heavy-tailed stable distributions \cite{SamTaq,EmMa}. In order to overcome some limitations of its areas of application, coming from  stationarity of its increments as well as constancy over time of its self-similarity exponent, Stoev and Taqqu introduced in \cite{stoev2004stochastic} an extension of LFSM, called Linear Multifractional Stable Motion (LMSM), in which the Hurst parameter becomes a function $H(\cdot)$ depending on the time variable $t$. Similarly to LFSM, the tail heaviness of the marginal distributions of LMSM is determined by $\al$; also, under some conditions, its self-similarity is governed by $H(\cdot)$ and its path roughness is closely related to $H(\cdot)-1/\al$. Namely, it was shown in \cite{stoev2004stochastic} that $H(t_0)$ is the self-similarity exponent of LMSM at a time $t_0\neq 0$; moreover, very recently, it was established in \cite{hamonier2012lmsm}, that the quantities $\min_{t\in I} H(t)-1/\al$, and $H(t_0)-1/\al$, are respectively the uniform H\"older exponent of LMSM on a compact interval $I$, and its local H\"older exponent at $t_0$. 

The main goal of our article, is to construct, using wavelet coefficients of LMSM, strongly consistent (i.e. almost surely convergent) statistical estimators of $\min_{t\in I} H(t)$, $H(t_0)$, and $\al$; our estimation results, are obtained when $\al\in (1,2)$, and, $H(\cdot)$ is a H\"older function smooth enough, with values 
in a compact subinterval $[\underline{H},\overline{H}]$ of $(1/\al,1)$.
\end{abstract}

\medskip

      {\it Running head}:  Statistical inference for multifractional stable process  \\



\section{Introduction and statement of the main results}
\label{sec:Introestim}
Let $\al \in (0,2)$ \footnote{Notice that in the sequel we will restrict to $\al\in (1,2)$.} and $\{\Za{s}\,:\, s\in\R\}$ a symmetric $\al$-stable ($\sas$) L\'evy process with c\`adl\`ag paths defined on a probability space $(\O,\F,\PR)$, chosen\footnote{A careful inspection of the article \cite{hamonier2012lmsm}, shows that such a choice for $\O$, is possible.}, without loss of generality, in such a way that results concerning the typical path behavior of the continuous versions, defined below, of the Linear Multifractional Stable Motion (LMSM) $\{Y(t)\,:\, t\in\R\}$ and the $\sas$ random field $\{X(u,v):(u,v)\in \R\times (1/\al,1)\}$ generating it, hold, not only for almost all paths but also for all of them, in other words for every $\o\in\O$. 

Linear Fractional Stable Motion (LFSM) of stability parameter $\al$ and Hurst parameter $H\in (0,1)$, is the self-similar $\sas$ process with stationary increments, 
denoted\footnote{Even if $\al$ is an important parameter of LFSM, for the sake of simplicity, we prefer to denote this process by $\{X_{H}(t)\,:\, t\in\R\}$ instead of $\{X_{\al,H}(t)\,:\, t\in\R\}$.} by $\{X_{H}(t)\,:\, t\in\R\}$, and defined, for each $t\in\R$, as the $\sas$ stochastic integral:
\begin{equation}
\label{aeq:defLFSM}
X_{H}(t)=\int_{\R} \left\{(t-s)_{+}^{H-1/\alpha} - (-s)_{+}^{H-1/\alpha}\right\} \Za{ds},
\end{equation}
where, for all real numbers $x$ and $\kappa$, 
\begin{equation}
\label{h:eq:pospart}
(x)_+^{\kappa}=
\left\{
\begin{array}{l}
x^{\kappa}, \,\mbox{if $x\in (0,+\infty)$,}\\
\\
0,\, \mbox{else.}
\end{array}
\right.
\end{equation}
Roughly speaking, it can be viewed, as a fractional primitive or a fractional derivative of $\{\Za{s}\,:\, s\in\R\}$, depending on whether $H-1/\al>0$ or $H-1/\al<0$; on one hand, it reduces to the latter L\'evy process when $H=1/\al$, on the other hand, it becomes the usual Gaussian Fractional Brownian Motion (FBM), when $\al=2$. Thus, LFSM is one of the most natural extensions of FBM, to the setting of heavy-tailed stable distributions; two classical references on it and many other stochastic models with infinite variance are \cite{SamTaq,EmMa}. On one hand, similarly to $\{\Za{s}\,:\, s\in\R\}$, the tail heaviness of the marginal distributions of LFSM, is determined by $\al$; on the other hand, similarly to FBM, the self-similarity property of the finite-dimensional distributions of LFSM is governed by $H$; more precisely, for each fixed positive real number $a$, one has, 
$$
\{X_{H}(at)\,:\,t\in\R\} \stackrel{\mbox{{\tiny fdd}}}{=} \{a^H X_{H}(t)\,:\,t\in\R\}, 
$$
where $\stackrel{\mbox{{\tiny fdd}}}{=}$ means that the equality holds in the sense of the finite-dimensional distributions; therefore, the self-similarity exponent of LFSM is, at any time $t_0$, equal to $H$.

The fact that LFSM has stationary increments and a self-similarity exponent constant in time, restricts its areas of application: many real-life signals fail to satisfy these two properties. Therefore, Stoev and Taqqu \cite{stoev2004stochastic} have introduced a non stationary increments extension of LFSM called Linear Multifractional Stable Motion (LMSM), which, roughly speaking, consists in making the Hurst parameter of LFSM, to be dependent on the time variable $t$. More precisely, LMSM, denoted by $\{Y(t)\,:\,t\in\R\}$, is defined, for each $t\in\R$, as,
\begin{equation}\label{m:def:lmsm}
Y(t)=X(t,H(t)),
\end{equation}
where $\{X(u,v):(u,v)\in \R\times (0,1)\}$, is the $\sas$ random field, such that for every $(u,v)\in \R\times (0,1)$,
\begin{equation}
\label{h:def:champX}
X(u,v)=\int_{\R}\Big\{ (u-s)_+^{v-1/\alpha} - (-s)_+^{v-1/\alpha} \Big\} \Za{ds}.
\end{equation}
Notice that for each fixed $H\in (0,1)$, the process $X(\cdot,H)=\{X(t,H)\,:\, t\in\R\}$ is $X_H$ the LFSM of Hurst parameter $H$. Also notice that, when $\al=2$, then the process $\{Y(t)\,:\,t\in\R\}$ reduces to the Gaussian Multifractional Brownian Motion (MBM), which is the most known multifractional process. LMSM, indeed provides a model whose self-similarity exponent depends on time and thus may change over time; namely, as shown in \cite{stoev2004stochastic}, when $H(t)-H(t_0)=o\big(|t-t_0|^{H(t_0)}\big)$ at some fixed $t_0\ne 0$,
then the LMSM $\{Y(t)\,:\,t\in\R\}$, is at $t_0$, locally asymptotically self-similar of exponent $H(t_0)$, moreover the tangent process (this notion was introduced and studied in \cite{falconer2002tangent,falconer2003local}) is the LFSM $X(\cdot,H(t_0))$; more precisely, one has, 
\begin{equation}
\label{eqa:sseM}
\left\{ a^{-H(t_0)} (Y(t_0+a u)-Y(t_0)) \,:\,u\in\R\right\} \xrightarrow[a \rightarrow 0,\,a>0]{\mbox{{\tiny fdd}}} \{ X(u,H(t_0))\,:\,u\in\R\}. 
\end{equation}

Let us now make a few recalls concerning path continuity and roughness of LMSM. Observe that the fact that LFSMs with $H\le 1/\al$ do not have a version with continuous paths (see \cite{SamTaq,EmMa}), clearly implies that the field $\{X(u,v):(u,v)\in \R\times (0,1)\}$ itself, cannot have such a version. Yet, it has been shown in \cite{hamonier2012lmsm}, that when $\al$ belongs\footnote{From now on, we assume that $\al\in (1,2)$.} to $(1,2)$ and $(u,v)$ is restricted to $\R\times (1/\al,1)$, then one can construct, through random wavelet-type series, a version of $\{X(u,v):(u,v)\in \R\times (1/\al,1)\}$, also denoted by $\{X(u,v):(u,v)\in \R\times (1/\al,1)\}$, whose paths are continuous functions and satisfy several other nice properties; a quite useful one among them, is that, for all fixed compact intervals $\mathcal{H}\subset (1/\al,1)$ and $\mathcal{I}\subset\R$, the paths are Lipschitz functions with respect to $v\in\mathcal{H}$, uniformly in $u\in\mathcal{I}$, namely one has,
\begin{equation}
\label{eqa:uniflipX}
\sup\left\{\frac{\big|X(u,v_1)-X(u,v_2)\big|}{|v_1-v_2|}\,:\, u\in\mathcal{I}\mbox{ and } (v_1,v_2)\in \mathcal{H}^2\right\}<+\infty, 
\end{equation}
with the convention that $0/0=0$. In view of (\ref{m:def:lmsm}), it is clear that the LMSM $\{Y(t)\,:\,t\in\R\}$ has continuous paths, as long as its parameter $H(\cdot)$ is a continuous function with values in $(1/\al,1)$; this will be the case in all the sequel, moreover, for the sake of simplicity, we will even assume that the range of $H(\cdot)$ is included in a compact subinterval of $(1/\al,1)$, denoted by $[\underline{H},\overline{H}]$. We now turn to path roughness of LMSM, it can classically be measured through H\"older exponents. Recall that for each compact interval $I\subset\R$ with non-empty interior, and every $\ga\in [0,1]$, the H\"older space $\Ch^{\ga}(I)$ is defined as, 
$$
\Ch^{\ga}(I)=\left\{f:I\rightarrow\R\,:\,\sup_{t_1,t_2\in I}\frac{|f(t_1)-f(t_2)|}{|t_1-t_2|^{\ga}}<+\infty\right\}.
$$
Also recall that $\rho_{g}^{\mbox{{\tiny unif}}}(I)$ the uniform (or global) H\"older exponent over $I$, of a continuous function $g:\R\rightarrow\R$, is defined as, 
\begin{equation}
\label{aeq:defuhe}
\rho_{g}^{\mbox{{\tiny unif}}}(I)=\sup\left\{\ga\in [0,1]: g\in \Ch^{\ga}(I)\right\},
\end{equation}
and, $\rho_{g}^{\mbox{{\tiny unif}}}(t_0)$, its local (or uniform pointwise) H\"older exponent at an arbitrary $t_0\in\R$, is defined as,
\begin{equation}
\label{aeq:deflhe} 
\rho_{g}^{\mbox{{\tiny unif}}}(t_0)=\sup\left\{\rho_{g}^{\mbox{{\tiny unif}}}\big([M_1,M_2]\big): M_1\in\R,\, M_2\in\R \mbox{ and } M_1<t_0<M_2\right\}.
\end{equation}
In the case where $g$ is a LMSM path, these two exponents are respectively denoted by $\rho_{Y}^{\mbox{{\tiny unif}}}(I)$ and $\rho_{Y}^{\mbox{{\tiny unif}}}(t_0)$.
It has been shown in \cite{hamonier2012lmsm} that, under the condition, 
\begin{equation}
\label{aeq:condYuhe}
\rho_{H}^{\mbox{{\tiny unif}}}(I)>1/\al,
\end{equation}
one has 
\begin{equation}
\label{aeq:valYuhe}
\rho_{Y}^{\mbox{{\tiny unif}}}(I)=\min_{t\in I} H(t)-1/\al,
\end{equation}
and, under the condition, 
\begin{equation}
\label{aeq:condYlhe}
\rho_{H}^{\mbox{{\tiny unif}}}(t_0)>1/\al,
\end{equation}
one has 
\begin{equation}
\label{aeq:valYlhe}
\rho_{Y}^{\mbox{{\tiny unif}}}(t_0)=H(t_0)-1/\al.
\end{equation}
The equalities (\ref{aeq:valYuhe}) and (\ref{aeq:valYlhe}), show that the quantities, $\min_{t\in I} H(t)-1/\al$, and $H(t_0)-1/\al$, provide important information concerning global and local path roughness of LMSM; moreover, as we already pointed it out (see (\ref{eqa:sseM})), $H(t_0)$ is its self-similarity exponent at $t_0$, and $\al$ determines tail heaviness of its marginal distributions. Therefore, it seems natural to look for statistical estimation procedures of $\min_{t\in I} H(t)$, $H(t_0)$ and $\al$; this is what we intend to do in our article, by making use of wavelet coefficients of LMSM. Let us mention that in the case of LFSM, the problem of the estimation of the constant Hurst parameter, has already been studied in several works (see for example \cite{abry1999estimation,delbeke98,delbeke2000stochastic,pipiras2007bounds,SPT2002}) 
and strongly consistent wavelet estimators have been obtained; yet, this problem becomes more tricky, in the more general case of LMSM, because the Hurst parameter becomes changing with time and the increments of the process are no longer stationary. Also it is worth mentioning that, recently, using an approach different from us, Le Gu\'evel \cite{leguevel2010}, has been able to construct, for a class of multistable processes which includes LMSM,  statistical estimators of their time changing Hurst and stability parameters; however, it is not clear that the latter estimators be strongly consistent (almost surely convergent) and how they may allow to estimate $\min_{t\in I} H(t)$. Before ending this paragraph, let us mention that in the case of Gaussian MBM, strongly consistent estimators of $\min_{t\in I} H(t)$ and $H(t_0)$, have been first obtained in \cite{benassi1998identifying}, through generalized quadratic variations of this process, and by imposing to its parameter $H(\cdot)$ to be a $C^1$ function. After this first article on this issue, the works on it, have focused on the estimation of $H(t_0)$. Many results of \cite{benassi1998identifying}, have been sharpened in \cite{coeurjolly2005identification,coeurjolly2006erratum}; among other things, it has been shown that it is enough to impose to $H(\cdot)$ to be a H\"older function with an exponent strictly bigger than all the values of $H(\cdot)$. An extensive study on the estimation of $H(t_0)$, in a more general Gaussian setting than that of MBM, has been done in \cite{bardet2010nonparametric}; among other things, it shows that the new increment ratio method can be an alternative to the classical quadratic variations one. A wavelet estimator of this quantity in the framework of MBM, has been given in \cite{ayache2011prostate}, and has turned out to be efficient in detection of prostate cancer tumors on mr images. Last but not least, the estimation of $H(t_0)$ in a setting of noisy MBM, has been studied in \cite{peng11}.

In order to state the two main results of our article, one needs to introduce some notations.

\begin{itemize}
\item[$\bullet$] The sequence of the $\sas$ random variables $\{d_{j,k}\,:\,(j,k)\in\Z^2\}$, denotes the wavelet coefficients of the LMSM $\{Y(t)\,:\,t\in\R\}$; they are defined by,
\begin{equation}\label{m:djk}
d_{j,k}=2^j \int_{\R} Y(t) \psi(2^jt-k) dt=2^j\int_{\R} X(t,H(t)) \psi(2^jt-k) dt,
\end{equation}
the last equality results from (\ref{m:def:lmsm}). It is worth noticing that, one only imposes to the analyzing wavelet $\psi$ three weak assumptions:
\begin{enumerate}
\item it is not the zero function, 
\item it is continuous on $\R$, and vanishes outside of the interval $[0,1]$, in other words,
\begin{equation}\label{m:eq:support}
\mathrm{supp}\, \psi \subseteq [0,1],
\end{equation}
\item its first two moments vanish, i.e. 
\begin{equation}\label{m:2moments}
\int_{\R} \psi(s) ds = \int_{\R} s\psi(s) ds = 0.
\end{equation}
\end{enumerate}
Let us underline that there is no need that $\{2^{j/2} \psi(2^j\cdot-k): (j,k)\in\Z^2\}$ be an orthonormal wavelet basis of $\LL(\R)$, or even a frame in this space; information on orthonormal or biorthogonal wavelet bases, as well as on frames, can be found in \cite{Dau92}, for instance.
\item[$\bullet$] $\{I_j\,;\,j\in\Z_+\}$ is an arbitrary non-increasing\footnote{That is, $I_{j}\subseteq I_{j-1}$, for all $j\in\N$.} (in the sense of the inclusion) sequence of compact intervals, satisfying for each $j\in\Z_+$,
\begin{equation}\label{h:cond2:Ij}
|I_j|=\mathrm{diam}(I_j) =\sup\big\{|x_1-x_2|\,:\, (x_1,x_2)\in I_{j}^2\big\}\ge 2^{1-j/2}.
\end{equation}
\item[$\bullet$] For all $j\in\Z_+$, one can sets,  
\begin{equation}\label{h:def1:Hj}
\underline{H}_j=\min_{t\in I_j} H(t);
\end{equation}
in view of the continuity of $H(\cdot)$, and of the compactness of $I_j$, this definition of $\underline{H}_j$ makes sense, moreover, there exists, at least one $\mu_j\in I_j$, such that,
\begin{equation}\label{h:def2:Hj}
H(\mu_j)=\underline{H}_j.
\end{equation}
\item[$\bullet$] At last, $\{\nu_j \,;\,j\in\Z_+ \}$ denotes the sequence of the finite sets of indices $k$, defined by, 
\begin{equation}\label{h:def:nuj}
\nu_j=\Big\{ k\in\Z : \big[k2^{-j},(k+1)2^{-j}\big]\subseteq I_j \Big\};
\end{equation}
moreover, $n_j$ is the number of the elements of $\nu_j$, namely,
\begin{equation}\label{h:def:cardnuj}
n_j=\mathrm{card }(\nu_j).
\end{equation} 
Observe that, in view of (\ref{h:cond2:Ij}),  for all $j\in\Z_+$, one has 
\begin{equation}
\label{h:cond1:nj}
n_j\ge \big[2^{j/2}\big]\ge 1, 
\end{equation}
where $[\cdot]$ denotes the integer part function.
\end{itemize}
$ $

Now, we are in position to state the two main results of this article.
\begin{Theo}\label{h:thm:main1}
Let $\beta \in (0,\al/4)$ be arbitrary and fixed.  For all $j\in\Z_+$, one denotes by $V_j$, the empirical mean, of order $\beta$, defined as, 
\begin{equation}\label{h:def:VJ}
V_j=\frac{1}{n_j} \sum_{k\in\nu_j} |d_{j,k}|^{\beta}.
\end{equation}
Assume that there exists $j_0\in\Z_+$ such that $\rho_{H}^{\mbox{{\tiny unif}}}(I_{j_0})$, the uniform H\"older exponent of $H(\cdot)$ over $I_{j_0}$, satisfies,
\begin{equation}
\label{Am:cond:pholder1}
\rho_{H}^{\mbox{{\tiny unif}}}(I_{j_0}) >\max_{t\in I_{j_0}} H(t).
\end{equation}
Then, one has, almost surely,
\begin{equation}\label{h:main:equation}
\lim_{j\rightarrow + \infty} \Bigg| \frac{\log_2(V_j)}{-j\beta} - \underline{H}_j \Bigg|=0.
\end{equation}
\end{Theo}
 
\begin{Rem}
\label{Remh:thm:main1}
\begin{enumerate}
\item Let $S$ be an arbitrary $\sas$ random variable, the condition $\beta \in (0,\al/4)$, implies that the second order moment of the random variable $|S|^\beta$ is finite; in the case where the parameter $\al\in (1,2)$ is unknown, then $\beta$ will be chosen in the interval $(0,1/4]\subset (0,\al/4)$. Note in passing, that it is possible to choose $\beta$ in the interval $[\al/4,\alpha/2)$, yet, the condition (\ref{h:cond2:Ij}) needs then to be strengthened, in order that Theorem~\ref{h:thm:main1} to be satisfied (more precisely, in order that (\ref{eqa:bc:lfgn1}) still results from (\ref{eqa:bc3}), (\ref{h:maj1:controle:bjlambda}) and (\ref{h:cond1:nj})).  
\item Observe that it follows from (\ref{aeq:defuhe}) and (\ref{Am:cond:pholder1}), that there exist two positive constants $c$ and $\rho_H$, satisfying, \begin{equation}\label{m:cond:holder}
\forall t_1,t_2\in I_{j_0}, \, |H(t_1)-H(t_2)| \le c |t_1-t_2|^{\rho_H},
\end{equation}
and 
\begin{equation}\label{m:cond:pholder}
1\ge \rho_H >\max_{t\in I_{j_0}} H(t).
\end{equation}
\item Assume that $I$ is a compact interval with non-empty interior and such that $\rho_{H}^{\mbox{{\tiny unif}}}(I) >\max_{t\in I} H(t)$. Taking $I_j=I$, for every $j\ge -2\log_2 (2^{-1}|I|)$, and else $I_j=L$, where $L$ is an arbitrary fixed compact interval such that $I\subseteq L$ and $|L|\ge 2$; then, it follows from Theorem~\ref{h:thm:main1}, that,
$
\frac{\log_2(V_j)}{-j\beta},
$
is a strongly consistent estimator (i.e. which converges almost surely ) of $\min_{t\in I} H(t)$.
\item Assume that $t_0$ is such that $\rho_{H}^{\mbox{{\tiny unif}}} (t_0)>H(t_0)$, and that $\{I_j\,;\,j\in\Z_+\}$ is an arbitrary non-increasing sequence, which satisfies (\ref{h:cond2:Ij}) as well as,
$$\{t_0\}=\bigcap_{j\in\Z_+} I_j;$$
then, it follows from Theorem~\ref{h:thm:main1} and from the continuity of $H(\cdot)$, that, 
$
\frac{\log_2(V_j)}{-j\beta},
$
is a strongly consistent estimator (i.e. which converges almost surely ) of $H(t_0)$.
\item  A careful inspection of the proof of Theorem~\ref{h:thm:main1}, shows that it remains valid in the Gaussian setting of Multifractional Brownian Motion (MBM); in this case, $\al=2$, one takes $\beta=2$, and $H(\cdot)$ is allowed to be with values in any compact subinterval of $(0,1)$. Even in this quite classical framework, the theorem provides two interesting results which were unknown so far, namely:
\begin{itemize}
\item it gives a strongly consistent wavelet estimator of $\min_{t\in I} H(t)$;
\item it shows that, it is possible (at least from a theoretical point of view) to estimate $H(t_0)$, even if the rate of convergence to zero, when $j$ goes to $+\infty$, of the diameters of the intervals $I_j$, is arbitrarily slow; in particular there is no need to impose to this rate to be a power function or a logarithmic function of $2^{-j}$, as it has be done in the previous literature.
\end{itemize}
\end{enumerate}
\end{Rem}
 
 \begin{Theo}
\label{m:Th:main}
Assume that $I$ is a compact interval with non-empty interior and such that, 
\begin{equation}
\label{Am:cond:pholder}
\rho_{H}^{\mbox{{\tiny unif}}}(I) >\max_{t\in I} H(t).
\end{equation}
Denote by $\underline{\widehat{H}}_j(I)$ the strongly consistent
estimator of $\min_{t\in I} H(t)$, introduced in Part~3 of Remark~\ref{Remh:thm:main1}; moreover, 
for each fixed positive integer $j\ge -2\log_2 (2^{-1}|I|)$, one sets,
\begin{equation}
\label{eq:eqa:defDg}
D_j=\max\big\{|d_{j,k}|\,:\, [2^{-j}k,2^{-j}(k+1)]\subseteq I\big\}
\end{equation}
and
\begin{equation*}
\label{m:def:estimal}
\widehat{\al}_j=\left(\underline{\widehat{H}}_j(I)+j^{-1}\log_{2} (D_j)\right)^{-1}.
\end{equation*}
Then, one has, almost surely,
\begin{equation*}\label{m:conv:estimal}
\widehat{\al}_j \xrightarrow[j\rightarrow+\infty]{a.s.} \al.
\end{equation*}
\end{Theo}

\begin{Rem}
\label{Rem:m:Th:main}
Theorem~\ref{m:Th:main} has already been obtained in \cite{hamonier2012lfsm}, in the less general case of LFSM.
\end{Rem}

The remaining of this article is structured in the following way: Section~\ref{sec:proofthm:main1} is devoted to the proof of Theorem~\ref{h:thm:main1} and Section~\ref{sec:proofm:Th:main} to that of Theorem~\ref{m:Th:main}. Before ending the present section, let us mention that in the sequel, we will make a rather frequent use of the following two classical properties (see \cite{SamTaq}) of $\sas$ random variables:
\begin{itemize}
\item if $S$ is an arbitrary $\sas$ random variable, then one has for all $\ga\in (0,\al)$,
\begin{equation}
\label{eqa:equivmoS}
\ESP(|S|^{\ga})=c(\ga)\|S\|_{\al}^{\ga},
\end{equation}
where the (finite) constant $c(\ga)$ only depends on $\ga$, and $\|S\|_{\al}$ denotes the scale parameter of~$S$;
\item moreover, when $S=\int_{\R} f(s)\Za{ds}$ for some $f\in L^{\al}(\R)$, then one has,
\begin{equation}
\label{eqa:scparamS}
\|S\|_{\al}^{\al}=\int_{\R}|f(s)|^\al ds.
\end{equation}
\end{itemize}

\section{Proof of Theorem~\ref{h:thm:main1}} 
\label{sec:proofthm:main1}
A natural strategy for proving Theorem~\ref{h:thm:main1}, variants of which have already been used several times in the literature, consists in the following two main steps:
\begin{enumerate}
\item to establish that, when $j$ goes to $+\infty$, the asymptotic behavior of the empirical mean $V_j$, is equivalent to that of its expectation $\ESP(V_j)$, namely one has,
\begin{equation}
\label{eqa:ratioVj}
 \frac{V_j}{\ESP \big( V_j \big)} \xrightarrow[j\rightarrow+\infty]{a.s.} 1;
 \end{equation}
 \item to show that (\ref{h:main:equation}) holds when $V_j$ is replaced by $\ESP(V_j)$.
 \end{enumerate}
 
Rather than directly working with wavelet coefficients $d_{j,k}$ (see (\ref{m:djk})), it is better to work with their approximations $\widetilde{d}_{j,k}$, defined for each $(j,k)\in\Z^2$, as,
\begin{equation}\label{h:tdjk}
\widetilde{d}_{j,k}=2^j \int_{\R} X(t,H(k2^{-j})) \psi(2^jt-k) dt;
\end{equation}
indeed, in view of \cite[Corollary 1]{delbeke2000stochastic} and of the fact that the process $\{X(t,H(k2^{-j}))\,:\,t\in\R\}$ is a LFSM of Hurst parameter $H(k2^{-j})$, the latter coefficients offer the advantage to have a stable stochastic integral representation which is 
rather convenient to handle. More precisely, for all $(j,k)\in\Z^2$, 
\begin{equation}\label{h:tdjk2}
\widetilde{d}_{j,k}\stackrel{a.s.}{=} 2^{-j(H(k2^{-j})-1/\al)} \int_{\R} \Phi_{\al}(2^js-k,H(k2^{-j})) \Za{ds},
\end{equation}
where $\Phi_{\al}$ is the real-valued function, defined, for every $(s,v)\in \R \times (1/\al,1)$, as,
\begin{equation}\label{h:def:phial}
\Phi_{\al}(s,v) = \int_{\R} (y-s)_+^{v-1/\al} \psi(y) dy;
\end{equation}
later (see Proposition~\ref{h:prop:proprietephial}), we will show that $\Phi_{\al}$ satisfies several nice properties. 

The following lemma provides, uniformly in $k\in\nu_j$, a control of the made error, when $d_{j,k}$ is replaced by $\widetilde{d}_{j,k}$.

\begin{Lem}\label{h:lem:ecartdjktdjk}
There exists a finite positive random variable $C$, such that for all $\omega \in \Omega$, and $j\in\Z_+$, one has, 
\begin{equation}\label{h:ecartdjktdjk}
\max_{k\in\nu_j}\big|d_{j,k}(\omega)-\widetilde{d}_{j,k}(\omega)\big| \le C(\omega) 2^{-j\rho_H},
\end{equation}
where $\rho_H$ has been introduced in Part~2 of Remark~\ref{Remh:thm:main1}.
\end{Lem}

\begin{proof}[Proof of Lemma \ref{h:lem:ecartdjktdjk}]
There is no restriction to assume that $j\ge j_0$, where $j_0$ has been introduced in the statement of Theorem~\ref{h:thm:main1}.
Putting together (\ref{m:djk}), (\ref{h:tdjk}), the change of variable $t=2^{-j}k+2^{-j}x$, and (\ref{m:eq:support}), it comes that, for all $k\in\nu_j$, 
\begin{align*}
& \big|d_{j,k}(\omega)-\widetilde{d}_{j,k}(\omega)\big| \\
& \le \int_{\R} \Big|X\big(2^{-j}k+2^{-j}x,H(2^{-j}k+2^{-j}x),\omega\big)-X\big(2^{-j}k+2^{-j}x,H(2^{-j}k),\omega\big)\Big| \big|\psi(x)\big| dx \\
& \le \int_{0}^{1} \sup_{u\in I_{j_0}} \Big|X\big(u,H(2^{-j}k+2^{-j}x),\omega\big)-X\big(u,H(k2^{-j}),\omega\big)\Big| \big|\psi(x)\big| dx, 
\end{align*}
where the last inequality results from (\ref{h:def:nuj}) and the inclusion $I_{j}\subseteq I_{j_0}$. Thus, using (\ref{eqa:uniflipX}), when $\mathcal{I}=I_{j_0}$ and $\mathcal{H}=[\underline{H},\overline{H}]$, one obtains that,
\begin{align*}
\big|d_{j,k}(\omega)-\widetilde{d}_{j,k}(\omega)\big| & \le C'(\omega) \|\psi\|_{\Li(\R)} \int_0^1 \big|H(2^{-j}k+2^{-j}x)-H(2^{-j}k)\big| dx \\
& \le C(\o)2^{-j\rho_H},
\end{align*}
where, the last inequality follows from (\ref{m:cond:holder}); notice that $C'$ and $C$ are finite positive random variables non depending on $(j,k)$. 
\end{proof}

Let us now consider $\widetilde{V}_j$, the empirical mean of order $\beta$, defined, for all $j\in\Z_+$, as,
\begin{equation}\label{h:definition:Vjtilde:lfgn1}
\widetilde{V}_j= \frac{1}{n_j}\sum_{k\in\nu_j} |\widetilde{d}_{j,k}|^{\beta}.
\end{equation}
The following two propositions correspond to the two main steps of the strategy for proving Theorem~\ref{h:thm:main1}, in which $V_j$ is replaced by $\widetilde{V}_j$. 

\begin{Prop}\label{h:prop:lfgn1}
One has almost surely, 
\begin{equation}\label{h:lfgn1}
\frac{\widetilde{V}_j}{\ESP \big( \widetilde{V}_j \big)} \xrightarrow[j\rightarrow+\infty]{a.s.} 1.
\end{equation}
\end{Prop}

\begin{Prop}\label{h:lem:control1}
One has,
\begin{equation}\label{eqa:h:control1}
\lim_{j\rightarrow + \infty} \Bigg| \frac{\log_2\big(\ESP(\widetilde{V_j})\big)}{-j\beta} - \underline{H}_j \Bigg|=0.
\end{equation}
\end{Prop}

Let us first focus on the proof of Proposition~\ref{h:prop:lfgn1}, later one will come back to that of Proposition~\ref{h:lem:control1}. Thanks to Borel-Cantelli Lemma, for showing that Proposition~\ref{h:prop:lfgn1} holds, it is sufficient to prove that, for any fixed real-number $\eta>0$, one has,
\begin{equation}
\label{eqa:bc:lfgn1}
\sum_{j=0}^{+\infty} \PR\Bigg( \Big| \frac{\widetilde{V}_j}{\ESP \big( \widetilde{V}_j \big)} -1  \Big| > \eta \Bigg)<+\infty;
\end{equation}
in order to obtain (\ref{eqa:bc:lfgn1}), one needs to derive a convenient upper bound for the probability,
$$
\PR\Bigg( \Big| \frac{\widetilde{V}_j}{\ESP \big( \widetilde{V}_j \big)} -1  \Big| > \eta \Bigg).
$$
Notice that Markov inequality, implies that,
\begin{equation}\label{h:markov:lfgn1}
\PR\Bigg( \Big| \frac{\widetilde{V}_j}{\ESP \big( \widetilde{V}_j \big)} -1  \Big| > \eta \Bigg) \le \eta^{-2} \frac{\var\big( \widetilde{V}_j \big)}{ \Big( \ESP( \widetilde{V}_j ) \Big)^2 };
\end{equation}
then, in view of the equalities,
\begin{equation}
\label{eqa:esperancetildeV}
\ESP \big( \widetilde{V}_j \big)=\frac{1}{n_j}\sum_{k\in \nu_j} \ESP\big(|\widetilde{d}_{j,k}|^{\beta}\big)
\end{equation}
and
\begin{equation}
\label{eqa:variancetildeV}
\var\big( \widetilde{V}_j \big) = \frac{1}{n_{j}^2}\sum_{(k,l)\in\nu_j\times \nu_j} \cov\big( |\widetilde{d}_{j,k}|^{\beta} ,  |\widetilde{d}_{j,l}|^{\beta} \big),
\end{equation}
it turns out that, for conveniently bounding,
 $$
\PR\Bigg( \Big| \frac{\widetilde{V}_j}{\ESP \big( \widetilde{V}_j \big)} -1  \Big| > \eta \Bigg),
$$
one needs to estimate, with some precision, $\ESP\big(|\widetilde{d}_{j,k}|^{\beta}\big)$ (or equivalently (see (\ref{eqa:equivmoS})) the scale parameter $\|\widetilde{d}_{j,k} \|_{\al}$ of $d_{j,k}$) and $\big|\cov\big( |\widetilde{d}_{j,k}|^{\beta} ,  |\widetilde{d}_{j,l}|^{\beta} \big)\big|$. The following two results provide rather sharp estimates of those quantities.

\begin{Lem}\label{h:lem:borne:scaletdjk}
There exist two constants $0<c_{2}\le c_1$, such that for every $(j,k)\in\Z^2$, one has,
\begin{equation}\label{h:borne:scaletdjk}
c_{2}2^{-jH(k2^{-j})} \le \|\widetilde{d}_{j,k} \|_{\al} \le c_{1} 2^{-jH(k2^{-j})},
\end{equation}
where $\|\widetilde{d}_{j,k} \|_{\al}$ is the scale parameter of the $\sas$ random variable $\widetilde{d}_{j,k}$.
\end{Lem}

\begin{Prop}\label{h:prop:maj1:cov}
Let $\lambda$ be the strictly positive real number, defined as,
\begin{equation}\label{h:def:lambdaopt}
\lambda= \min \left\{ \frac{\al}{2}\Big(2+\frac{1}{\al}-\overline{H}\Big)\,; (\al-1)\Big(2+\frac{1}{\al}-\overline{H}\Big)\right\}.
\end{equation}
Then, there exists a constant $c>0$, non depending on $(j,k,l)$, such that the inequality, 
\begin{equation}\label{h:maj1:cov}
\big| \cov( |\widetilde{d}_{j,k}|^{\beta},|\widetilde{d}_{j,l}|^{\beta} )\big| \le c 2^{-j\beta(H(k2^{-j})+H(l2^{-j}))} \big(1+|k-l|\big)^{-\lambda},
\end{equation}
holds for each $(j,k,l)\in\Z^3$.
\end{Prop}
\begin{Rem}
\label{Rem1:newcov}
The inequality (\ref{h:maj1:cov}) is completely sufficient for the purposes of our present article, however it is worth noticing that, a careful inspection of its proof, 
shows that it can be improved in the following way:
$$
\big| \cov( |\widetilde{d}_{j,k}|^{\beta},|\widetilde{d}_{j,l}|^{\beta} )\big| \le c 2^{-j\beta(H(k2^{-j})+H(l2^{-j}))} \big(1+|k-l|\big)^{-\lambda_{j,k}},
$$
where, $c>0$ is a constant non depending on $(j,k,l)$, and where 
$$
\lambda_{j,k}= \min \left\{ \frac{\al}{2}\Big(2+\frac{1}{\al}-\max\big\{H(k2^{-j}),H(l2^{-j})\big\}\Big)\,; (\al-1)\Big(2+\frac{1}{\al}-\max\big\{H(k2^{-j}),H(l2^{-j})\big\}\Big)\right\};
$$
if moreover, the analyzing wavelet $\psi$, has $N\ge 2$ vanishing moments, i.e.
$$
\int_{\R} \psi(s) ds = \int_{\R} s\psi(s) ds =\ldots=\int_{\R} s^{N-1}\psi(s) ds=0,
$$
then one can even take,
$$
\lambda_{j,k}= \min \left\{ \frac{\al}{2}\Big(N+\frac{1}{\al}-\max\big\{H(k2^{-j}),H(l2^{-j})\big\}\Big)\,; (\al-1)\Big(N+\frac{1}{\al}-\max\big\{H(k2^{-j}),H(l2^{-j})\big\}\Big)\right\}.
$$
\end{Rem}

Some preliminary results are required, for proving Lemma~\ref{h:lem:borne:scaletdjk} and Proposition~\ref{h:prop:maj1:cov}. 

\begin{Prop}\label{h:prop:proprietephial}
The function $\Phi_{\al}$ defined in (\ref{h:def:phial}), satisfies the three following "nice" properties: 
\begin{itemize}
\item[(i)] The function $\Phi_{\al}$ is continuous over $\R\times (1/\al,1)$.
\item[(ii)] The function $\Phi_{\al}$ is well-localized in $s$ uniformly in  $v\in [\underline{H},\overline{H}]$, more precisely, one has,
\begin{equation}
\label{h:local:phial}
\left\{
\begin{array}{l}
\supp(\Phi_{\al}(\cdot,v)) \subseteq ]-\infty,1],\,\mbox{for all fixed $v\in(1/\al,1)$,} \\
\mbox{and}\\
c_1 =\sup\left\{\big(1+|s|\big)^{2+1/\al-\overline{H}}\big|\Phi_{\al}(s,v) \big| \,:\, (s,v)\in ]-\infty,1]\times [\underline{H},\overline{H}]\right\}<+\infty.
\end{array}
\right.
\end{equation}
\item[(iii)] The finite nonnegative function $v \mapsto \| \Phi_{\al}(\cdot,v)\|_{\La(\R)}$ is continuous over $[\underline{H},\overline{H}]$;
moreover, one has, 
\begin{equation}\label{h:minphial}
c_2 =\min_{v\in [\underline{H},\overline{H}]} \| \Phi_{\al}(\cdot,v)\|_{\La(\R)} >0.
\end{equation}
\end{itemize}
\end{Prop}
Before proving Proposition \ref{h:prop:proprietephial}, let us mention that Part $(ii)$ of it, is a generalization of Proposition~2.1 in \cite{hamonier2012lfsm}.
\begin{proof}[Proof of Proposition \ref{h:prop:proprietephial}] First one shows that Part $(i)$ holds; namely, for each $(\tilde{u},\tilde{v})\in \R \times (1/\al,1)$, and for every sequence $\{(u_n,v_n):n\in\Z_+ \}$ of elements of $\R\times (1/\al,1)$ converging to $(\tilde{u},\tilde{v})$, one has,
$$ 
\Phi_{\al}(\tilde{u},\tilde{v})=\lim_{n\rightarrow +\infty} \Phi_{\al}(u_n,v_n),
$$
in other words (see (\ref{h:def:phial})),
\begin{equation} 
\label{eqa:leb11}
\int_{\R}  (y-\tilde{u})_+^{\tilde{v}-1/\al}  \psi(y) dy =\lim_{n\rightarrow  +\infty} \int_{\R} (y-u_n)_+^{v_n-1/\al} \psi(y) dy. 
\end{equation}
The equality (\ref{eqa:leb11}) will result from \tcd. One has, for each fixed $y\in\R$, 
\begin{equation}\label{h:eq1:phial}
\lim_{n\rightarrow + \infty} (y-u_n)_+^{v_n-1/\al} \psi(y) = (y-\tilde{u})_+^{\tilde{v}-1/\al} \psi(y),
\end{equation}
since the function $(s,v) \mapsto (y-s)_+^{v-1/\al} \psi(y)$ is continuous over $\R \times (1/\al,1)$; on the other hand, the triangle inequality and the inequality $0<v_n-1/\al<1/2$, for each $n\in\Z_+$, imply that
\begin{equation}\label{h:eq2:phial}
|(y-u_n)_+^{v_n-1/\al} \psi(y) | \le \Big(1+|y|+\sup_{n\in\Z_+} |u_n| \Big) |\psi(y)|.
\end{equation}
Thus, in view of (\ref{h:eq1:phial}), (\ref{h:eq2:phial}), and of the fact that the function $y\mapsto \Big(1+|y|+\sup_{n\in\Z_+} |u_n| \Big) |\psi(y)|$ belongs to $\Lun(\R)$, one is allowed to obtain (\ref{eqa:leb11}) by applying \tcd.

Let us now prove that Part $(ii)$ holds. It follows from (\ref{m:eq:support}) and (\ref{h:def:phial}), that one has, for each $(s,v)\in\R \times (1/\al,1)$, 
\begin{equation}\label{h:def:phial1}
\Phi_{\al}(s,v) = \int_{0}^{1} (y-s)_+^{v-1/\al}\psi(y) dy.
\end{equation}
Next, combining (\ref{h:eq:pospart}) with (\ref{h:def:phial1}), one gets, for all fixed $v\in (1/\al,1)$, that,
$$
\supp(\Phi_{\al}(\cdot,v)) \subseteq ]-\infty,1].
$$
Let us now turn to the proof of the finiteness of the quantity $c_1$ defined in (\ref{h:local:phial}). The equality (\ref{h:def:phial1}) easily implies that, 
$$
\sup\left\{\big(1+|s|\big)^{2+1/\al-\overline{H}}\big|\Phi_{\al}(s,v) \big| \,:\, (s,v)\in [-1,1]\times [\underline{H},\overline{H}]\right\} \le 4 \|\psi\|_{\Li(\R)}<+\infty;
$$
so, it is enough to show that
\begin{equation}
\label{eqah:local:phial}
\sup\left\{\big(1-s\big)^{2+1/\al-\overline{H}}\big|\Phi_{\al}(s,v) \big| \,:\, (s,v)\in ]-\infty,-1)\times [\underline{H},\overline{H}]\right\}<+\infty.
\end{equation}
One denotes by $\psi^{(-1)}$ and $\psi^{(-2)}$ the primitives of $\psi$, of order $1$ and $2$, defined, for all $z\in\R$, by
\begin{equation*}
\psi^{(-1)}(z)=\int_{-\infty}^{z} \psi(y) dy \,\mbox{ and } \,\psi^{(-2)}(z)=\int_{-\infty}^{z} \psi^{(-1)}(y) dy;
\end{equation*}
it follows from (\ref{m:2moments}) and (\ref{m:eq:support}), that they are compactly supported with supports contained in $[0,1]$, moreover, they clearly are continuous functions on the whole real line. Next, assuming that $(s,v)\in ]-\infty,-1)\times [\underline{H},\overline{H}]$ and making in (\ref{h:def:phial1}) two integrations by parts, one obtains that, 
\begin{equation*}
\Phi_{\al}(s,v)= (v-1/\al)(v-1/\al-1)\int_{0}^{1} (y-s)^{v-1/\al-2} \psi^{(-2)}(y) dy;
\end{equation*}
thus, the inequality $y-s\ge -s \ge 2^{-1}(1-s)$ for all $y\in[0,1]$, entails that,
\begin{align*}
|\Phi_{\al}(s,v)| & \le  2^{2+1/\al-v} |v-1/\al||v-1/\al-1| \|\psi^{(-2)}\|_{\Li(\R)} (1-s)^{v-1/\al-2} \nonumber \\
& \le 2 \|\psi^{(-2)}\|_{\Li(\R)} (1-s)^{\overline{H}-1/\al-2},
\end{align*}
which means that (\ref{eqah:local:phial}) is satisfied.

Let us now prove that Part $(iii)$ holds. Notice that, in view of Part $(ii)$, one knows that the function $v \mapsto \|\Phi_{\al}(\cdot,v)\|_{\La(\R)}$ is finite on $[\underline{H},\overline{H}]$. Proving that the latter function is continuous over this interval, boils down to show the continuity of $v\mapsto \|\Phi_{\al}(\cdot,v)\|_{\La(\R)}^\al=\int_{\R} |\Phi_{\al}(u,v)|^{\al} du$ over it; namely, for each $\tilde{v}\in [\underline{H},\overline{H}]$ and for every sequence $\{v_n\,:\, n\in\Z_+\}$ of elements of $[\underline{H},\overline{H}]$ converging to $\tilde{v}$, one has, 
\begin{equation}
\label{eqa:leb12}
 \int_{\R} |\Phi_{\al}(u,\tilde{v})|^{\al} du=\lim_{n\rightarrow +\infty} \int_{\R}|\Phi_{\al}(u,v_n)|^{\al} du.
\end{equation}
The equality (\ref{eqa:leb12}) will result from \tcd. On one hand, Part $(i)$ of the Proposition, implies, that for each fixed $u\in\R$,    
\begin{equation}\label{h:phial:maj3}
\lim_{n\rightarrow +\infty} |\Phi_{\al}(u,v_n)|^{\al} = |\Phi(u,\tilde{v})|^{\al}; 
\end{equation}
on the other hand, from Part $(ii)$ of it, entails that the inequality,
\begin{equation}\label{h:phial:maj4}
|\Phi_{\al}(u,v_n)|^{\al} \le c_1^{\al} (1+|u|)^{-\al(2+1/\al-\overline{H})},
\end{equation}
is satisfied, for all $u\in\R$ and every $n\in\Z_+$. Thus, in view of (\ref{h:phial:maj3}), (\ref{h:phial:maj4}), and of the fact that the function $u \mapsto (1+|u|)^{-\al(2+1/\al-\overline{H})}$  belongs to $\Lun(\R)$, one is allowed to obtain (\ref{eqa:leb12}) by applying \tcd.

Finally let us show that (\ref{h:minphial}) holds. It follows from the continuity of the function $v\mapsto \|\Phi_{\al}(\cdot,v)\|_{\La(\R)}$, and from the compactness of the interval $[\underline{H},\overline{H}]$, that there exists $v_0 \in [\underline{H},\overline{H}]$ such that $c_2= \|\Phi_{\al}(\cdot,v_0)\|_{\La(\R)}$. Suppose ad absurdum that $c_2=0$, then one has, for any $u\in\R$, $\Phi_{\al}(u,v_0)=0$; which is equivalent to say that for all $\xi\in\R$, 
\begin{equation}\label{h:ant:phial:maj4}
\widehat{\Phi}_{\al}(\xi,v_0)=0,
\end{equation}
where $\widehat{\Phi}_{\al}(\cdot,v_0)$ denotes the Fourier transform of $\Phi_{\al}(\cdot,v_0)$, defined for each $\xi\in\R$, as,
$
\widehat{\Phi}_{\al}(\xi,v_0)=\int_{\R} e^{-i\xi u}\,\Phi_{\al}(u,v_0) du. 
$
Next, let $\Gamma$ be the usual "gamma function" defined, for all $x\in (0,+\infty)$, by 
$
\Gamma(x)=\int_{0}^{+\infty} t^{x-1} e^{-t} dt.
$
Noticing that the function
$
\big(\Gamma(1+v_0-1/\al)\big)^{-1} \Phi_{\al}(\cdot,v_0),
$
is the right-sided fractional derivative of order  $1+v_0-1/\al$ of the wavelet $\psi$; it follows that (see \cite{samk}), for any $\xi \in \R\setminus\{0\}$,
\begin{equation}
\label{h:ant:phial:maj5}
\widehat{\Phi}_{\al}(\xi,v_0)= \Gamma(1+v_0-1/\al)\frac{e^{i\,\sign(\xi)(1+v_0-1/\al)\frac{\pi}{2}}\widehat{\psi}(\xi)}{|\xi|^{1+v_0-1/\al}}.
\end{equation}
Finally, combining (\ref{h:ant:phial:maj4}) with (\ref{h:ant:phial:maj5}), one obtains that $\widehat{\psi}$ is the zero function, which contradicts the assumption that $\psi$ is not the zero function.
\end{proof}

Now, we are in position to show that Lemma \ref{h:lem:borne:scaletdjk} holds.
\begin{proof}[Proof of Lemma \ref{h:lem:borne:scaletdjk}]
Using (\ref{h:tdjk2}), (\ref{eqa:scparamS}), and the change of variable $u=2^js-k$, one gets that,
\begin{align*}
\|\widetilde{d}_{j,k}\|_{\al}^{\al}& = 2^{-j\al H(k2^{-j})+j} \int_{\R} |\Phi_{\al}(2^js-k,H(k2^{-j}))|^{\al} ds \\
& = 2^{-j\al H(k2^{-j})} \int_{\R} |\Phi_{\al}(u,H(k2^{-j}))|^{\al} du.
\end{align*}
So, it comes that
\begin{equation}\label{h:scaletdjk}
\|\widetilde{d}_{j,k}\|_{\al} = 2^{-jH(k2^{-j})} \|\Phi_{\al}(\cdot,H(k2^{-j}))\|_{\La(\R)}.
\end{equation}
Let us set,
\begin{equation}
\label{eqa:lemEdjk}
c_{1}=\max_{v\in [\underline{H},\overline{H}]} \|\Phi_{\al}(\cdot,v)\|_{\La(\R)}\, \mbox{ and } \, c_{2}=\min_{v\in [\underline{H},\overline{H}]} \|\Phi_{\al}(\cdot,v)\|_{\La(\R)};
\end{equation}
in view of Part $(iii)$ of Proposition~\ref{h:prop:proprietephial}, one has $0<c_2\le c_1<+\infty$. Finally, combining (\ref{h:scaletdjk}) with (\ref{eqa:lemEdjk}),
one deduces that (\ref{h:borne:scaletdjk}) holds.
\end{proof}

Let us now focus on the proof of Proposition~\ref{h:prop:maj1:cov}. One denotes by $\{\widetilde{d}_{j,k}^{\nor} : (j,k)\in\Z^2\}$ the sequence of the normalized\footnote{That is the scale parameters are equal to $1$.} $\sas$ random variables, defined as, 
\begin{equation}\label{h:tdjknor}
\widetilde{d}_{j,k}^{\nor} = \frac{\widetilde{d}_{j,k}}{\|\widetilde{d}_{j,k}\|_{\al}};
\end{equation}
observe that, for all $(j,k,l)\in\Z^3$, one has, 
\begin{equation}\label{h:egalite:covariance}
\big|\cov\big( |\widetilde{d}_{j,k}|^{\beta}, |\widetilde{d}_{j,l}|^{\beta} \big)\big| =\|\widetilde{d}_{j,k}\|_{\al}^{\beta} \|\widetilde{d}_{j,l}\|_{\al}^{\beta}\, \big|\cov\big( |\widetilde{d}_{j,k}^{\nor}|^{\beta}, |\widetilde{d}_{j,l}^{\nor}|^{\beta} \big)\big|.
\end{equation}
Thus, in view of (\ref{h:borne:scaletdjk}), for showing that (\ref{h:maj1:cov}) holds, it is sufficient to prove that,
\begin{equation}
\label{eqa:cov30}
\sup \left\{\big (1+|k-l|\big)^{\lambda}\big|\cov\big( |\widetilde{d}_{j,k}^{\nor}|^{\beta}, |\widetilde{d}_{j,l}^{\nor}|^{\beta} \big)\big|\,:\, (j,k,l)\in\Z^3\right\}<+\infty,
\end{equation}
or equivalently that, there exists a constant $Q_0\in\Z_+$, such that, 
\begin{equation}
\label{eqa:cov31}
\sup \left\{\big (1+|k-l|\big)^{\lambda}\big|\cov\big( |\widetilde{d}_{j,k}^{\nor}|^{\beta}, |\widetilde{d}_{j,l}^{\nor}|^{\beta} \big)\big|\,:\, (j,k,l)\in\Z^3\mbox{ and } |k-l|\ge Q_0\right\}<+\infty.
\end{equation}
Observe that (\ref{eqa:cov31}) reduces to (\ref{eqa:cov30}) when $Q_0=0$; the fact that they are equivalent when $Q_0\ge 1$, mainly results from, Cauchy-Schwarz inequality, which entails that,
\begin{align*}
& \sup \left\{\big (1+|k-l|\big)^{\lambda}\big|\cov\big( |\widetilde{d}_{j,k}^{\nor}|^{\beta}, |\widetilde{d}_{j,l}^{\nor}|^{\beta} \big)\big|\,:\, (j,k,l)\in\Z^3\mbox{ and } |k-l|< Q_0\right\}\\
& \le Q_0^{\lambda}\big(\var(|X_0|^{\beta})\big)^2<+\infty,
\end{align*}
where $X_0$ denotes an arbitrary $\sas$ random variable satisfying $\|X_0\|_{\al}=1$.
It is worth noticing that, in view of (\ref{h:tdjk2}) and of (\ref{h:tdjknor}), $\widetilde{d}_{j,k}^{\nor}$ can be expressed as the $\sas$ stochastic integral,
\begin{equation}
\label{h:def:fjk}
\widetilde{d}_{j,k}^{\nor} = \int_{\R} f_{j,k}(s) \Za{ds},
\end{equation}
where $f_{j,k}$ is the function belonging to the unit sphere of $L^\al (\R)$, defined for all $s\in\R$, as,
\begin{equation}
\label{h:fjk}
f_{j,k}(s)= \theta_{j,k}\,2^{j/\al}\Phi_{\al}(2^js-k,H(k2^{-j})), 
\end{equation}
with
\begin{equation}\label{h:thetajk}
\theta_{j,k}=\frac{2^{-jH(k2^{-j})}}{\|\widetilde{d}_{j,k}\|_{\al}}.
\end{equation}
Observe that (\ref{h:borne:scaletdjk}) implies that, 
\begin{equation}
\label{eqa:tetasup}
\theta_{\mbox{{\tiny sup}}}=\sup\left\{\theta_{j,k}\,:\,(j,k)\in\Z^2\right\}<+\infty,
\end{equation}
and
$$
\theta_{\mbox{{\tiny inf}}}=\inf\left\{\theta_{j,k}\,:\,(j,k)\in\Z^2\right\}>0.
$$
In order to establish the existence of a constant $Q_0\in\Z_+$, such that (\ref{eqa:cov31}) holds, one will use the following lemma, which is a consequence of 
\cite[Theorem 2.4]{pipiras2007bounds}.
\begin{Lem} 
\label{leme:PTA08}
\cite{pipiras2007bounds} Let $f^1$ and $f^2$ be two functions of $L^\al (\R)$ satisfying, for some $c'\in (0,1)$ and $c''\in (0,1)$, the following two Conditions $(\mathcal{A}_1)$ and $(\mathcal{A}_2)$:
$$
 \|f^1\|_{\La(\R)}^{\al/2} \|f^2\|_{\La(\R)}^{\al/2} - \int_{\R} |f^1(s)f^2(s)|^{\al/2} ds \ge c' \|f^1\|_{\La(\R)}^{\al/2}\|f^2\|_{\La(\R)}^{\al/2};
\eqno(\mathcal{A}_1)
$$
for every $(x,y)\in\R^2$,
$$ 
\|xf^1+yf^2\|_{\La(\R)}^{\al} \ge c'' \Big( \|xf^1\|_{\La(\R)}^{\al} + \|yf^2\|_{\La(\R)}^{\al} \Big).
\eqno(\mathcal{A}_2)
$$
Then, there exists a constant $c>0$, only depending on $c'$, $c''$, $\beta$, $\|f^1\|_{\La(\R)}$ and $\|f^2\|_{\La(\R)}$, such that one has,
$$
\left|\cov\left(\Big|\int_{\R} f^{1}(s)\Za{ds}\Big|^{\beta},\Big|\int_{\R} f^{2}(s)\Za{ds}\Big|^{\beta}\right)\right|\le c\Big( [f^1,f^2]_1+[f^1,f^2]_2 \Big)
$$
where, 
\begin{equation}\label{h:def:crochet2}
[f^1,f^2]_2= \int_{\R} |f^1(s)f^2(s)|^{\al/2} ds,
\end{equation}
\begin{equation}\label{h:def:crochet1}
[f^1,f^2]_1= [f^1,f^2]_1^{*} + [f^2,f^1]_1^{*},
\end{equation}
and 
\begin{equation}\label{h:def:crochet11}
[f^1,f^2]_1^{*}= \int_{\R} |f^1(s)|^{\al-1} |f^2(s)| ds.
\end{equation}
\end{Lem} 

\begin{Rem}
\label{remnew:condsa}
The following two lemmas, respectively show that the functions $f_{j,k}$ and $f_{j,l}$, defined through (\ref{h:fjk}), satisfy, uniformly in $(j,k,l)\in\Z^3$,
the Conditions $(\mathcal{A}_1)$ and $(\mathcal{A}_2)$ in Lemma~\ref{leme:PTA08}, provided that $|k-l|\ge Q_0$, where $Q_0\in\Z_+$ is a large enough constant.
\end{Rem}

\begin{Lem}\label{h:lem:hypA1}
Let $Q_0 \in\Z_+$ be arbitrary but large enough. Then, there exists a constant $c_{1}\in (0,1)$, non depending on $(j,k,l)$, such that the inequality, 
\begin{equation}\label{h:hypA1}
1 - \int_{\R} |f_{j,k}(s)f_{j,l}(s)|^{\al/2} ds \ge c_{1},
\end{equation}
holds, for all $(j,k,l)\in\Z^3$ satisfying $|k-l|\ge Q_0$.
\end{Lem}

\begin{Lem}\label{h:lem:hypA2}
Let $Q_0 \in\Z_+$ be arbitrary but large enough. Then, there exists a constant $c_{2}\in (0,1)$, non depending on $z$, neither $(j,k,l)$, such that the inequality,
\begin{equation}\label{h:equi:hypA2}
\|f_{j,k}+zf_{j,l}\|_{\La(\R)}^{\al} \ge c_{2} \Big( 1 + |z|^{\al} \Big).
\end{equation}
holds, for all $z\in [-1,1]$ and for each $(j,k,l)\in\Z^3$ satisfying $|k-l|\ge Q_0$.
\end{Lem}

\begin{proof}[Proof of Remark~\ref{remnew:condsa}] The equalities
\begin{equation}
\label{eqa1:new:condsa}
\|f_{j,k}\|_{L^\al (\R)}=\|f_{j,l}\|_{L^\al (\R)}=1,
\end{equation}
clearly entail that the inequality (\ref{h:hypA1}), can be rewritten as,
$$
\|f_{j,k}\|_{L^\al (\R)}^{\al/2}\|f_{j,l}\|_{L^\al (\R)}^{\al/2}- \int_{\R} |f_{j,k}(s)f_{j,l}(s)|^{\al/2} ds \ge c_{1}\|f_{j,k}\|_{L^\al (\R)}^{\al/2}\|f_{j,l}\|_{L^\al(\R)}^{\al/2};
$$
which shows that the functions $f_{j,k}$ and $f_{j,l}$, satisfy, uniformly in $(j,k,l)\in\Z^3$ (i.e. the constant $c_1$ does not depend on $(j,k,l)$),
the Condition $(\mathcal{A}_1)$ in Lemma~\ref{leme:PTA08}. Next observe that the Condition $(\mathcal{A}_2)$ in the lemma, does not become weaker, when one restricts to 
the $(x,y)\in\R^2$ such that $0<|y|\le |x|$; indeed, it clearly holds when $x=0$ or $y=0$, on the other hand the indices $k$ and $l$, play in it symmetric roles, thus they can be interchanged. So, let us assume that $x$ and $y$ are two arbitrary real numbers satisfying $0<|y|\le |x|$, then setting in (\ref{h:equi:hypA2})
$z=y/x$, and using (\ref{eqa1:new:condsa}), one gets that,
$$ 
\|xf_{j,k}+yf_{j,l}\|_{\La(\R)}^{\al} \ge c_2 \Big( \|xf_{j,k}\|_{\La(\R)}^{\al} + \|yf_{j,l}\|_{\La(\R)}^{\al} \Big);
$$
which shows that the functions $f_{j,k}$ and $f_{j,l}$, satisfy, uniformly in $(j,k,l)\in\Z^3$ (i.e. the constant $c_2$ does not depend on $(j,k,l)$),
the Condition $(\mathcal{A}_2)$ in Lemma~\ref{leme:PTA08}.
\end{proof}

The proofs of Lemmas~\ref{h:lem:hypA1}~and~\ref{h:lem:hypA2}, mainly rely on the following result, which also, will allow us later, to bound $[f_{j,k},f_{j,l}]_2$ (see (\ref{h:def:crochet2}) for the definition of $[\cdot,\cdot]_2$).

\begin{Lem}\label{h:lem:maj:phi2jkl}
For all $(j,k,l)\in\Z^3$, one sets,
\begin{equation}
\label{h:def:phi2jkl}
\varphi_1(j,k,l) = \int_{\R} \big|\Phi_{\al}(u-k,H(k2^{-j})) \Phi_{\al}(u-l,H(l2^{-j}))\big|^{\al/2} du.
\end{equation}
Let $\epsilon>0$ be arbitrarily small and fixed; there exists a constant $c>0$ which does not depend on $(j,k,l)$, such that the inequality,
\begin{equation}
\label{h:maj:phi2jkl}
\varphi_1(j,k,l)\le c\big(1+|k-l|\big)^{-\al/2(2+1/\al-\overline{H})},
\end{equation}
holds for every $(j,k,l)\in\Z^3$.
\end{Lem}

In order to show that the previous lemma holds, one needs the following result.

\begin{Lem}\label{h:lem:maj1:rq}
Let $\delta$ and $\gamma$ be  two nonnegative real numbers satisfying,
\begin{equation}\label{h:hyp:maj1:rq}
\max\{\delta,\gamma\}>1.
\end{equation}
For each $q\in\Z$, one sets
\begin{equation}\label{h:def:rq}
r_q(\delta,\gamma)=\int_{\R} \big(1+|u-q|\big)^{-\delta} \big(1+|u|\big)^{-\gamma} du.
\end{equation}
Then, one has,
\begin{equation}\label{h:maj1:rq}
\sup_{q\in\Z} \left\{ \big(1+|q|\big)^{\min\{\delta,\gamma\}} \,r_q(\delta,\gamma) \right\} < +\infty.
\end{equation}
\end{Lem}


\begin{proof}[Proof of Lemma \ref{h:lem:maj1:rq}]
First observe that easy computations allow to prove that $r_q(\delta,\gamma)=r_q(\gamma,\delta)$, so one can assume that $\delta=\max\{\delta,\gamma\}$ i.e. 
$\gamma=\min\{\delta,\gamma\}$; also one can assume, without any restriction that $q$ is an arbitrary but non-vanishing integer. Then, making in (\ref{h:def:rq}), the change of variable $x=q^{-1}u$, one gets that,
\begin{equation}
\label{eqa1:maj1:rq}
r_q(\delta,\gamma)=|q|\int_{\R} \big(1+|qx-q|\big)^{-\delta} \big(1+|qx|\big)^{-\gamma} dx=|q|^{1-\delta-\gamma}\int_{\R} \big(|q|^{-1}+|x-1|\big)^{-\delta} \big(|q|^{-1}+|x|\big)^{-\gamma} dx.
\end{equation}
Moreover, one has, for some (finite) constant $c>0$ (large enough), non depending on $q$,
$$
s_{1,q}(\delta,\gamma)=\int_{-\infty}^{-1} \big(|q|^{-1}+1-x\big)^{-\delta} \big(|q|^{-1}-x\big)^{-\gamma} dx\le \int_{1}^{+\infty}x^{-\delta-\gamma} dx\le c,
$$
$$
s_{2,q}(\delta,\gamma)=\int_{-1}^{0} \big(|q|^{-1}+1-x\big)^{-\delta} \big(|q|^{-1}-x\big)^{-\gamma} dx\le\int_{-1}^{0} \big(|q|^{-1}-x\big)^{-\gamma} dx\le c\big(1+ |q|^{\gamma-1}+\log(|q|)\big),
$$
$$
s_{3,q}(\delta,\gamma)=\int_{0}^{1/2} \big(|q|^{-1}+1-x\big)^{-\delta} \big(|q|^{-1}+x\big)^{-\gamma} dx\le 2^{\delta}\int_{0}^{1/2} \big(|q|^{-1}+x\big)^{-\gamma} dx\le c\big(1+|q|^{\gamma-1}+\log(|q|)\big),
$$
$$
s_{4,q}(\delta,\gamma)=\int_{1/2}^{1} \big(|q|^{-1}+1-x\big)^{-\delta} \big(|q|^{-1}+x\big)^{-\gamma} dx\le 2^{\gamma}\int_{1/2}^{1}\big(|q|^{-1}+1-x\big)^{-\delta}dx\le c\big(1+|q|^{\delta-1}\big),
$$
$$
s_{5,q}(\delta,\gamma)=\int_{1}^{2} \big(|q|^{-1}+x-1\big)^{-\delta} \big(|q|^{-1}+x\big)^{-\gamma} dx\le \int_{1}^{2} \big(|q|^{-1}+x-1\big)^{-\delta}  dx \le c\big(1+|q|^{\delta-1}\big),
$$
and
$$
s_{6,q}(\delta,\gamma)=\int_{2}^{+\infty} \big(|q|^{-1}+x-1\big)^{-\delta} \big(|q|^{-1}+x\big)^{-\gamma} dx\le \int_{1}^{+\infty}x^{-\delta-\gamma} dx\le c;
$$
note in passing that $\log(|q|)$ can be removed when $\gamma\ne 1$. Putting together (\ref{eqa1:maj1:rq}), the previous inequalities, the fact that $\delta=\max\{\delta,\gamma\}>1$ and the equality,
$$
\int_{\R} \big(|q|^{-1}+|x-1|\big)^{-\delta} \big(|q|^{-1}+|x|\big)^{-\gamma} dx=\sum_{l=1}^{6} s_{l,q}(\delta,\gamma),
$$
one gets (\ref{h:maj1:rq}).
\end{proof}

\begin{proof}[Proof of Lemma \ref{h:lem:maj:phi2jkl}]
Making in (\ref{h:def:phi2jkl}) the change of variables $x=u-l$, and using (\ref{h:local:phial}), it follows that,
$$
\varphi_1(j,k,l)\le c_1^{\al} \int_{\R} (1+|x-(k-l)|)^{-\al/2(2+1/\al-\overline{H})}  (1+|x|)^{-\al/2(2+1/\al-\overline{H})} dx,
$$
where $c_1$ denotes the same constant as in (\ref{h:local:phial}). Then, for bounding from above the latter integral, the inequality $\al/2(2+1/\al-\overline{H})>1$, allows us to apply Lemma~\ref{h:lem:maj1:rq}, when $\delta=\gamma=\al/2(2+1/\al-\overline{H})$ and $q=k-l$; thus, one gets (\ref{h:maj:phi2jkl}).
\end{proof}

\begin{proof}[Proof of Lemma \ref{h:lem:hypA1}]
In view of (\ref{h:fjk}) and (\ref{eqa:tetasup}), one has,
\begin{equation*}
\int_{\R} |f_{j,k}(s)f_{j,l}(s)|^{\al/2} ds \le \theta_{\mbox{{\tiny sup}}}^{\al}\,2^j \int_{\R}  |\Phi_{\al}(2^js-k,H(k2^{-j}))\Phi_{\al}(2^js-l,H(l2^{-j}))|^{\al/2} ds; 
\end{equation*}
thus, using the change of variable $u=2^js$, (\ref{h:def:phi2jkl}) and (\ref{h:maj:phi2jkl}), one gets that,
\begin{equation}
\label{h:maj1:fjkfjl}
\int_{\R} |f_{j,k}(s)f_{j,l}(s)|^{\al/2} ds \le \theta_{\mbox{{\tiny sup}}}^{\al}\,\varphi_1(j,k,l) \le  c_{2}' \big(1+|k-l|\big)^{-\al/2(2+1/\al-\overline{H})}, 
\end{equation}
where $c_{2}'>0$ is a constant non depending on $(j,k,l)$. Then, assuming that $Q_0\in\Z_+$ in such that $c_{2}'(1+Q_0)^{-\al/2(2+1/\al-\overline{H})}\le 1/2$, one gets, when  $|k-l|\ge Q_0$, that,
 \begin{equation}\label{h:maj2:fjkfjl}
\int_{\R} |f_{j,k}(s)f_{j,l}(s)|^{\al/2} ds \le 1/2.
\end{equation}
Finally, setting $c_{1}=1/2$, it follows from (\ref{h:maj2:fjkfjl}), that (\ref{h:hypA1}) holds.
\end{proof}

\begin{proof}[Proof of Lemma \ref{h:lem:hypA2}] 
First observe that, thanks to the triangle inequality, one has,
\begin{align*}
& \big\|f_{j,k}+zf_{j,l}\big\|_{\La(\R)}^{\al} = \int_{\R} \left|\Big(f_{j,k}(s)+zf_{j,l}(s)\Big)^2\right|^{\al/2} ds \\
& \ge \int_{\R} \left| \Big( |f_{j,k}(s)|-|z| |f_{j,l}(s)| \Big)^2 \right|^{\al/2} ds \\
& =\int_{\R} \Big| \big|f_{j,k}(s)\big|^2+ z^2 \big|f_{j,l}(s)\big|^2 - 2 |z| \big|f_{j,k}(s)f_{j,l}(s)\big |    \Big|^{\al/2} ds;
\end{align*}
thus, the inequality 
$ |\zeta_1-\zeta_2|^{\al/2} \ge |\zeta_1^{\al/2}-\zeta_2 ^{\al/2}| \ge \zeta_{1}^{\al/2}-\zeta_{2}^{\al/2}$ for all $(\zeta_1,\zeta_2) \in \R_+^2$, implies that
\begin{align}\label{h:min3:hypA2}
& \|f_{j,k}+zf_{j,l}\|_{\La(\R)}^{\al} \nonumber \\
& \ge \int_{\R} \Big( \big|f_{j,k}(s)\big|^2+z^2 \big|f_{j,l}(s)\big|^{2} \Big)^{\al/2} ds - 2^{\al/2} |z|^{\al/2} \int_{\R} \big|f_{j,k}(s)f_{j,l}(s)\big|^{\al/2} ds.
\end{align}
On the other hand, the fact that $x\mapsto x^{\al/2}$ is a concave function over $\R_+$, entails that, for all $s\in\R$,
\begin{align}\label{h:min4:hypA2}
& \Big( \big|f_{j,k}(s)\big|^2+z^2 \big|f_{j,l}(s)\big|^{2} \Big)^{\al/2} = 2^{\al/2} \Big(\frac{1}{2} \big|f_{j,k}(s)\big|^2 +\frac{z^2}{2} \big|f_{j,l}(s)\big|^2 \Big)^{\al/2} \nonumber \\
& \ge 2^{\al/2-1} \Big( \big|f_{j,k}(s)\big|^{\al}+ |z|^{\al} \big|f_{j,l}(s)\big|^{\al} \Big).
\end{align}
Next, putting together, (\ref{eqa1:new:condsa}), the inequality $|z|\le 1$, (\ref{h:min3:hypA2}) and (\ref{h:min4:hypA2}), one gets, that,
\begin{equation}\label{h:min5:hypA2}
 \big\|f_{j,k}+zf_{j,l}\big\|_{\La(\R)}^{\al} \ge 2^{\al/2-1} (1+|z|^{\al})- 2^{\al/2} \int_{\R} |f_{j,k}(s)f_{j,l}(s)|^{\al/2} ds.
\end{equation}
Finally, setting $c_2=2^{\al/2-2}$, and assuming that $Q_0\in\Z_+$ is such that 
$c_{2}'(1+Q_2)^{-\al/2(2+1/\al-\overline{H})}\le 2^{-2}$, where $c_{2}'$ is the same constant as in (\ref{h:maj1:fjkfjl}); (\ref{h:min5:hypA2}) and (\ref{h:maj1:fjkfjl}), then imply that,
\begin{align*}
\big\|f_{j,k}+zf_{j,l}\big\|_{\La(\R)}^{\al} & \ge 2^{\al/2-1} \big(1+|z|^{\al}\big)-2^{\al/2-2} = 2^{\al/2-2}\big(1+2|z|^{\al}\big) \\
& \ge c_{2} \big(1+|z|^{\al}\big),
\end{align*}
for each $(j,k,l)\in\Z^3$, satisfying $|k-l|\ge Q_0$.
\end{proof}

The following lemma will allow us to bound $[f_{j,k},f_{j,l}]_1^{*}$ (see (\ref{h:def:crochet11}) for the definition of $[\cdot,\cdot]_1^{*}$). 
\begin{Lem}\label{h:lem:maj:phi1jkl}
For every $(j,k,l)\in\Z^3$, one sets, 
\begin{equation}
\label{h:def:phi1jkl}
\varphi_2(j,k,l)= \int_{\R} \big|\Phi_{\al}(u-k,H(k2^{-j}))\big|^{\al-1} \big|\Phi_{\al}(u-l,H(l2^{-j}))\big| du.
\end{equation}
There exists a constant $c>0$, non depending on $(j,k,l)$, such that the inequality,
\begin{equation}
\label{h:maj1:phi1jkl}
 \varphi_2 (j,k,l)\le c\big(1+|k-l|\big)^{-(\al-1)(2+1/\al-\overline{H})}, 
\end{equation}
holds for every $(j,k,l)\in\Z^3$.
\end{Lem}

\begin{proof}[Proof of Lemma \ref{h:lem:maj:phi1jkl}]
Making in (\ref{h:def:phi1jkl}) the change of variable $x=u-l$, and using (\ref{h:local:phial}), it follows that, 
$$
\varphi_2 (j,k,l)\le c_1^\al \int_{\R} (1+|x-(k-l)|)^{-(\al-1)(2+1/\al-\overline{H})} (1+|x|)^{-(2+1/\al-\overline{H})} dx, 
$$
where $c_1$ is the same constant as in (\ref{h:local:phial}). Then, for bounding from above the latter integral, the inequality $2+1/\al-\overline{H}>1$, allows us to apply Lemma~\ref{h:lem:maj1:rq}, when $\delta=(\al-1)(2+1/\al-\overline{H})$, $\gamma=2+1/\al-\overline{H}$, and $q=k-l$; thus, one gets (\ref{h:maj1:phi1jkl}).
\end{proof}

Now, we are in position to prove Proposition~\ref{h:prop:maj1:cov}.

\begin{proof}[Proof of Proposition \ref{h:prop:maj1:cov}] As, one has already pointed it out, it remains to show that (\ref{eqa:cov31}) holds, for some fixed $Q_0\in\Z_+$, large enough. Lemmas~\ref{h:lem:hypA1}~and~\ref{h:lem:hypA2}, allow one to apply Lemma~\ref{leme:PTA08}, for bounding $\big|\cov\big( |\widetilde{d}_{j,k}^{\nor}|^{\beta}, |\widetilde{d}_{j,l}^{\nor}|^{\beta} \big)\big|$; thus one gets,
\begin{equation}\label{h:maj1:covanorma}
\big|\cov\big( |\widetilde{d}_{j,k}^{\nor}|^{\beta}, |\widetilde{d}_{j,l}^{\nor}|^{\beta} \big)\big| \le c_{1} \Big( [f_{j,k},f_{j,l}]_1+[f_{j,k},f_{j,l}]_2 \Big),
\end{equation}
where $c_{1}>0$ is a constant non depending on $(j,k,l)$, since $\|\widetilde{d}_{j,k}^{\nor}\|_{\al}=\|\widetilde{d}_{j,k}^{\nor}\|_{\al}=1$.

Let us now derive a convenient upper bound for $[f_{j,k},f_{j,l}]_1= [f_{j,k},f_{j,l}]_1^{*} + [f_{j,l},f_{j,k}]_1^{*}$. Taking in (\ref{h:def:crochet11}), $f^1=f_{j,k}$ and $f^2=f_{j,l}$, and using (\ref{h:fjk}) and (\ref{eqa:tetasup}), one obtains that, 
$$
[f_{j,k},f_{j,l}]_1^{*} \le \theta_{\mbox{{\tiny sup}}}^{\al}\, 2^j\int_{\R} \big|\Phi_{\al}(2^js-k,H(k2^{-j}))\big|^{\al-1} \big|\Phi_{\al}(2^js-l,H(l2^{-j}))\big| ds;
$$
then, the change of variable $u=2^js$, and (\ref{h:def:phi1jkl}), yield to,
\begin{equation}\label{h:maj1:crochet11}
[f_{j,k},f_{j,l}]_1^{*} \le  \theta_{\mbox{{\tiny sup}}}^{\al}\varphi_2(j,k,l).
\end{equation}
Similarly to (\ref{h:maj1:crochet11}), one can show that 
\begin{equation}\label{h:maj2:crochet11}
[f_{j,l},f_{j,k}]_1^{*} \le \theta_{\mbox{{\tiny sup}}}^{\al}\varphi_2 (j,l,k).
\end{equation}
Next combining (\ref{h:def:crochet1}) and (\ref{h:maj1:crochet11}) with (\ref{h:maj2:crochet11}), one gets that, 
\begin{equation}\label{h:maj1:crochet1}
[f_{j,k},f_{j,l}]_1 \le \theta_{\mbox{{\tiny sup}}}^{\al} \Big( \varphi_2(j,k,l)+\varphi_2(j,l,k) \Big).
\end{equation}
Let us now derive convenient upper bound for $[f_{j,k},f_{j,l}]_2$. Taking in (\ref{h:def:crochet2}) $f^1=f_{j,k}$ and $f^2=f_{j,l}$, and using (\ref{h:fjk}) and (\ref{eqa:tetasup}), one obtains that,  
$$
[f_{j,k},f_{j,l}]_2 \le \theta_{\mbox{{\tiny sup}}}^{\al}\,2^j\int_{\R} \big|\Phi_{\al}(2^js-k,H(k2^{-j}))\Phi_{\al}(2^js-l,H(l2^{-j}))\big|^{\frac{\al}{2}} ds;
$$
then, the change of variable $u=2^js$, and (\ref{h:def:phi2jkl}), yield to,
\begin{equation}\label{h:maj1:crochet2}
[f_{j,k},f_{j,l}]_2 \le \theta_{\mbox{{\tiny sup}}}^{\al} \varphi_1(j,k,l).
\end{equation}
Finally, putting together (\ref{h:maj1:crochet1}), (\ref{h:maj1:crochet2}), (\ref{h:maj1:phi1jkl}), (\ref{h:maj:phi2jkl}) and (\ref{h:def:lambdaopt}), it follows that, (\ref{eqa:cov31}) holds, for some fixed $Q_0\in\Z_+$, large enough.
\end{proof}

Let us now come back to the proof of Proposition~\ref{h:prop:lfgn1}.
\begin{proof}[Proof of Proposition~\ref{h:prop:lfgn1}]
One has already seen that it is sufficient to show that (\ref{eqa:bc:lfgn1}) holds. Thanks to (\ref{h:markov:lfgn1}), (\ref{eqa:esperancetildeV}), (\ref{eqa:variancetildeV}), Lemma~\ref{h:lem:borne:scaletdjk} and Proposition \ref{h:prop:maj1:cov}, one knows that there is a constant $c>0$, non depending on $\eta$ and $j$, such that,
\begin{equation}
\label{eqa:bc3}
\PR\Bigg( \Big| \frac{\widetilde{V}_j}{\ESP \big( \widetilde{V}_j \big)} -1  \Big| > \eta \Bigg)\le c \,\eta^{-2} B_j^{\lambda},
\end{equation}
where 
\begin{equation}\label{h:def:bjlambda}
B_j^{\lambda}=\left (\sum_{(k,l)\in \nu_j\times\nu_j} 2^{-j\beta (H(k2^{-j})+H(l2^{-j}))} \Big(1+|k-l|\Big)^{-\lambda}\right) \left(\sum_{k\in\nu_j} 2^{-j\beta H(k2^{-j})} \right)^{-2}.
\end{equation}
Thus, (\ref{eqa:bc:lfgn1}) results from (\ref{eqa:bc3}), the following lemma, (\ref{h:def:cardnuj}), (\ref{h:cond1:nj}), and the inequality,
$$
2\beta(\overline{H}-\underline{H})<1/2.
$$
\end{proof}

\begin{Lem}\label{h:lem:controle:bjlambda}
Let $\widetilde{\lambda} \in (0,\lambda)$ be arbitrary and fixed, and let $j_0$ be as in the statement of Theorem~\ref{h:thm:main1}. There exists a constant $c>0$, which does not depend on $j$, such that for all  $j\ge\max\{1,j_0\}$, one has,
\begin{equation}\label{h:maj1:controle:bjlambda}
B_j^{\lambda} \le c\left\{ j^{1/\widetilde{\lambda}}\, n_j^{-1}\, 2^{j 2\beta (\overline{H}-\underline{H})}+ j^{-\lambda/ \widetilde{\lambda}} \right\};
\end{equation}
recall that $n_j$ has been defined in (\ref{h:def:cardnuj}), also recall that $\underline{H}=\inf_{x\in\R} H(x)$ and $\overline{H}=\sup_{x\in\R} H(x)$.
\end{Lem}

\begin{proof}[Proof of Lemma \ref{h:lem:controle:bjlambda}] For any fixed integer $j\ge 1$, let $\Delta_j$ and $\Delta_j^c$ be the subsets of $\nu_j \times \nu_j$, defined as,
\begin{equation}\label{h:definition:deltaj}
\Delta_j = \left\{(k,l)\in\nu_j \times \nu_j \, : \, |k-l|\le j^{1/ \widetilde{\lambda}}\right\}
\end{equation}
and
\begin{equation}\label{h:definition:deltajcpl}
\Delta_j^c = \left\{ (k,l)\in\nu_j \times \nu_j \, : \, |k-l|> j^{1/ \widetilde{\lambda}}\right\};
\end{equation}
observe that,
\begin{equation}
\label{h:ant:uniondelta}
\nu_j \times \nu_j=\Delta_j\cup \Delta_j^c,
\end{equation}
where the union is disjoint. For each fixed $k\in\nu_j$, one denotes by $\Delta_j(k)$ the subset of $\nu_j$ defined as,
\begin{equation}\label{h:definition:deltajk}
\Delta_j(k)= \Big\{ l\in\nu_j \,: \, (k,l)\in\Delta_j\Big\}=\Big\{l\in\nu_j \, : \, |k-l|\le j^{1/ \widetilde{\lambda}}\Big\}.
\end{equation}
Observe that,
\begin{equation}\label{h:maj:card:deltajk}
\mathrm{card }\big( \Delta_j(k) \big ) \le 2 j^{1/ \widetilde{\lambda}}+1\le 3 j^{1/ \widetilde{\lambda}}.
\end{equation}
The first inequality in (\ref{h:maj:card:deltajk}), results from the fact that there are at most $[d]+1$ integers belonging to a compact interval of an arbitrary diameter $d$; the second inequality in it, is a straightforward consequence of the first one, since $j^{1/ \widetilde{\lambda}}\ge 1$.

It follows from (\ref{h:definition:deltajcpl}), that 

\begin{align}\label{h:majoration1:controle:bjlambda}
& \sum_{(k,l)\in \Delta_j^c} 2^{-j\beta (H(k2^{-j})+H(l2^{-j}))} (1+|k-l|)^{-\lambda} \le j^{-\lambda/ \widetilde{\lambda}} \sum_{(j,k)\in \Delta_j^c} 2^{-j\beta (H(k2^{-j})+H(l2^{-j}))}  \nonumber \\
& \le j^{-\lambda/ \widetilde{\lambda}} \sum_{(j,k)\in \nu_j\times \nu_j} 2^{-j\beta (H(k2^{-j})+H(l2^{-j}))} \le j^{-\lambda/ \widetilde{\lambda}} \Bigg( \sum_{k\in\nu_j} 2^{-j\beta H(k2^{-j})} \Bigg)^2,
\end{align}
with the convention that $\sum_{(k,l)\in \Delta_j^c}\cdots=0$, when $\Delta_j^c$ is the empty set.

Let us now determine a convenient upper bound, for the sum,
$$
\sum_{(k,l)\in\Delta_j} 2^{-j\beta (H(k2^{-j})+H(l2^{-j}))} \big(1+|k-l|\big)^{-\lambda},
$$
with the convention that it equals $0$, when $\Delta_j$ is the empty set. First observe that, there exists a constant $c_{1}>0$, non depending on $j$ and $(k,l)$, such that for all $j\ge \max\{1,j_0\}$ and $(k,l)\in\Delta_j$, one has,
\begin{equation}\label{h:majoration2:controle:bjlambda}
2^{-j\beta (H(k2^{-j})+H(l2^{-j}))} \le c_{1}  2^{-j2 \beta H(k2^{-j})};
\end{equation}
the latter inequality holds, since $2^{-j}\Delta_j\subseteq I_j\times I_j \subseteq I_{j_0}\times I_{j_0}$, and thus, combining (\ref{m:cond:holder}) with (\ref{h:definition:deltaj}), it comes that,  
\begin{align*}
& 2^{-j\beta (H(k2^{-j})+H(l2^{-j}))} =  2^{-j2 \beta H(k2^{-j})}\, 2^{j\beta (H(k2^{-j})-H(l2^{-j}))} \\
& \le 2^{-j2 \beta H(k2^{-j})} \, 2^{j\beta |H(k2^{-j})-H(l2^{-j})|}  \le 2^{-j2 \beta H(k2^{-j})} \, 2^{j\beta c_2  |k2^{-j}-l2^{-j}|^{\rho_H}} \\
& \le 2^{-j2 \beta H(k2^{-j})} \exp\Big ((\log 2)\beta c_2\, 2^{-j\rho_H}j^{1+\rho_H/ \widetilde{\lambda}}\Big)\le c_{1}  2^{-j2 \beta H(k2^{-j})},
\end{align*}
where $c_2$ denotes the constant $c$ in (\ref{m:cond:holder}), and,
$$
c_{1}= \sup_{j\ge 1} \left\{\exp\Big((\log 2)\beta c_2 \,2^{-j\rho_H}j^{1+\rho_H/ \widetilde{\lambda}}\Big)\right\} < +\infty.
$$
Next, putting together (\ref{h:majoration2:controle:bjlambda}), (\ref{h:definition:deltaj}), (\ref{h:definition:deltajk}), and (\ref{h:maj:card:deltajk}), one gets that,
for all $j\ge \max\{1,j_0\}$,
\begin{align}\label{h:majoration3:controle:bjlambda}
& \sum_{(k,l)\in\Delta_j} 2^{-j\beta (H(k2^{-j})+H(l2^{-j}))} \big(1+|k-l|\big)^{-\lambda} \le c_{1} \sum_{(k,l)\in\Delta_j} 2^{-j2 \beta H(k2^{-j})} \big(1+|k-l|\big)^{-\lambda} \nonumber \\
& = c_{1} \sum_{k\in\nu_j} 2^{-j2 \beta H(k2^{-j})} \Bigg( \sum_{l\in \Delta_j(k)} \big(1+|k-l|\big)^{-\lambda} \Bigg) \le 3 c_{1}\, j^{1/\widetilde{\lambda}} \sum_{k\in\nu_j} 2^{-j2 \beta H(k2^{-j})}  \nonumber \\
& \le 3 c_{1}\, j^{1/\widetilde{\lambda}}\, n_j\, 2^{-j2\beta \underline{H}},
\end{align}
with the convention that $\sum_{l\in \Delta_j(k)}\cdots=0$, when $\Delta_j(k)$ is the empty set; notice that the last inequality results from the inequality $H(k2^{-j})\ge \underline{H}$, for all $k\in\nu_j$, as well as from (\ref{h:def:cardnuj}).

Next, observe that, the inequality $H(k2^{-j})\le \overline{H}$, for all $k\in\nu_j$, and (\ref{h:def:cardnuj}), imply that,
\begin{equation}\label{h:majoration4:controle:bjlambda}
\Bigg( \sum_{k\in\nu_j} 2^{-j\beta H(k2^{-j})} \Bigg)^2 \ge n_j^2\, 2^{-j2\beta \overline{H}}.
\end{equation}
Finally, (\ref{h:maj1:controle:bjlambda}) results from (\ref{h:def:bjlambda}), (\ref{h:ant:uniondelta}), (\ref{h:majoration1:controle:bjlambda}), (\ref{h:majoration3:controle:bjlambda}) and (\ref{h:majoration4:controle:bjlambda}).
\end{proof}

Having proved Proposition~\ref{h:prop:lfgn1}, let us now show that Proposition~\ref{h:lem:control1} holds.

\begin{proof}[Proof of Proposition~\ref{h:lem:control1}] First notice that it follows from (\ref{eqa:esperancetildeV}), (\ref{eqa:equivmoS}) and Lemma~\ref{h:lem:borne:scaletdjk}, that (\ref{eqa:h:control1}) is equivalent to,
\begin{equation}\label{h:control1}
\lim_{j\rightarrow +\infty} \Bigg| \frac{\log_2\Big( n_j^{-1} \sum_{k\in\nu_j}2^{-j\beta H(k2^{-j})} \Big)}{-j\beta} - \underline{H}_j \Bigg|=0;
\end{equation}
so proving the proposition remains to showing that (\ref{h:control1}) holds.

Putting together (\ref{h:def1:Hj}), (\ref{h:def:nuj}) and (\ref{h:def:cardnuj}), one can easily obtain that,
\begin{equation}\label{h:lem:control1:maj1}
n_j^{-1} \sum_{k\in\nu_j} 2^{-j\beta H(k2^{-j})} \le 2^{-j\beta \underline{H}_j}.
\end{equation}
Let $j_0$ be as in the statement of Theorem~\ref{h:thm:main1}, and let $\mu_j$ be as in (\ref{h:def2:Hj}). For each $j\ge \max\{1,j_0\}$, one denotes by $\widetilde{\nu}_j(\mu_j)$, the set of indices $k$, defined by, 
\begin{equation}\label{h:def:tnujmuj}
\widetilde{\nu}_j(\mu_j)=\Big\{k\in\nu_j\,:\,k2^{-j}\in B(\mu_j,j^{-1/\rho_H}) \Big\},
\end{equation}
where $\rho_H$ has been introduced in Part~2 of Remark~\ref{Remh:thm:main1}, and where $B(\mu_j,j^{-1/\rho_H})$ is the ball,
\begin{equation}\label{h:def:Bmuj}
B(\mu_j,j^{-1/\rho_H})=\left\{x\in\R\,:\, |x-\mu_j| \le j^{-1/\rho_H}\right\}.
\end{equation}
Let us show that there are $j_1\ge \max\{1,j_0\}$, and a constant $c_1 >0$, such that one has, for all $j\ge j_1$,
\begin{equation}\label{h:min1:card:tnuj}
\mathrm{card }( \widetilde{\nu}_j(\mu_j) ) \ge c_1 2^j \min\left\{ \frac{|I_j|}{2}, j^{-1/\rho_H} \right\}.
\end{equation}
The compact interval $I_j$ can be expressed as $I_j=[r_j,z_j]$, where $r_j<z_j$ are two real numbers, in addition, one assumes that,   
\begin{equation}\label{h:zjminusmuj}
z_j-\mu_j=\min\{ \mu_j-r_j,z_j-\mu_j \};
\end{equation}
the case where the minimum is reached in $\mu_j-r_j$ can be treated similarly. In view of (\ref{h:zjminusmuj}), one gets that, 
\begin{equation*}
\mu_j-r_j \ge z_j-\mu_j \,\Longleftrightarrow\, 2(\mu_j-r_j) \ge z_j-r_j \,\Longleftrightarrow\, \mu_j-r_j \ge \frac{z_j-r_j}{2}= \frac{|I_j|}{2};
\end{equation*} 
which implies that,
\begin{equation}\label{h:inclusion1Ij}
\left[ \mu_j-\frac{|I_j|}{2}, \mu_j \right] \subseteq I_j,
\end{equation}
and, as a consequence, that,  
\begin{equation*}
\left[ \mu_j-\min\left\{ \frac{|I_j|}{2} , j^{-1/\rho_H} \right\}, \mu_j \right] \subseteq I_j\cap B(\mu_j,j^{-1/\rho_H}).
\end{equation*}
Thus, it follows from (\ref{h:def:nuj}) and (\ref{h:def:tnujmuj}), that, 
$$
\left\{k\in\Z\,:\, \big[2^{-j}k,2^{-j}(k+1)\big]\in \left[ \mu_j-\min\left\{ \frac{|I_j|}{2} , j^{-1/\rho_H} \right\}, \mu_j \right]\right\}\subseteq \widetilde{\nu}_j(\mu_j);
$$
the latter inclusion and (\ref{h:cond2:Ij}), entail that (\ref{h:min1:card:tnuj}) holds, provided that $j_1$ is large enough. Moreover, in view of (\ref{h:cond2:Ij}), (\ref{h:def:nuj}) and (\ref{h:def:cardnuj}), when $j_1$ is big enough, one has, for every $j\ge j_1$,
\begin{equation}\label{h:borne1:nj}
c_2 2^j\, |I_j| \le n_j \le c_3 2^j \,|I_j|,
\end{equation}
where $0<c_2\le c_3$ are two constants non depending on $j$. Next combining (\ref{h:min1:card:tnuj}) with (\ref{h:borne1:nj}), one obtains that, for all $j\ge j_1$,
\begin{equation}
\label{h:min1:quotient:card}
\frac{\mathrm{card } ( \widetilde{\nu}_j(\mu_j) )}{n_j} \ge  \frac{ c_1 2^j \min\left\{ \frac{|I_j|}{2}, j^{-1/\rho_H} \right\}}{c_3 2^j |I_j|} = c_1 c_{3}^{-1} \min\left\{\frac{1}{2}; \frac{j^{-1/\rho_H}}{|I_j|} \right\}\ge c_4 j^{-1/\rho_H},
\end{equation}
where the last inequality results from the inclusion $I_j\subseteq I_0$, and where the strictly positive constant $c_4$ does not depend on $j$. Next observe that there exists a constant $c_5>0$, non depending on $j$ and $k$, such that the inequality
\begin{equation}\label{h:control1:min1}
2^{-j\beta H(k2^{-j})} \ge c_5 2^{-j\beta \underline{H}_j },
\end{equation}
holds, for all $j\ge j_1$ and each $k\in \widetilde{\nu}_j(\mu_j)$; indeed, putting together (\ref{h:def2:Hj}), the inclusion $I_{j}\subseteq I_{j_0}$, (\ref{m:cond:holder}) and  (\ref{h:def:tnujmuj}), one gets, for any integer $j\ge j_1$ and any $k\in \widetilde{\nu}_j(\mu_j)$, that
\begin{equation*}
2^{-j\beta(H(k2^{-j})-\underline{H}_j)} = 2^{-j\beta |H(k2^{-j})-\underline{H}_j|} \ge 2^{-j\beta c_6 |k2^{-j}-\mu_j|^{\rho_H}} \ge c_5 >0,
\end{equation*}
where $c_6$ denotes the constant $c$ in (\ref{m:cond:holder}), and $c_5=2^{-\beta c_6}$. Next, setting $c_7=c_4 c_5$, then, the inclusion $\widetilde{\nu}_j(\mu_j) \subseteq \nu_j$,  (\ref{h:min1:quotient:card}) and (\ref{h:control1:min1}), imply that
\begin{equation}\label{h:lem:control1:min2}
n_j^{-1} \sum_{k\in\nu_j} 2^{-j\beta H(k2^{-j})} \ge n_j^{-1} \sum_{k\in\widetilde{\nu}_j(\mu_j)} 2^{-j\beta H(k2^{-j})} \ge c_{7}\, j^{-1/\rho_H}\, 2^{-j\beta \underline{H}_{j}};
\end{equation}
which in turn entails that,
\begin{equation}\label{h:lem:control1:minlog}
\log_2 \Big( n_j^{-1} \sum_{k\in\nu_j} 2^{-j\beta H(k2^{-j})} \Big) \ge -j\beta\underline{H}_j + \log_2(c_{7}\, j^{-1/\rho_H}).
\end{equation}
On the other hand, it results from (\ref{h:lem:control1:maj1}), that
\begin{equation}\label{h:lem:control1:majlog}
\log_2 \Big( n_j^{-1} \sum_{k\in\nu_j} 2^{-j\beta H(k2^{-j})} \Big) \le -j\beta \underline{H}_j.
\end{equation}
Finally, combining (\ref{h:lem:control1:minlog}) and (\ref{h:lem:control1:majlog}) with the fact that $$\lim_{j\rightarrow +\infty} \frac{\log_2\,(c_{7}\, j^{-1/\rho_H})}{j} =0,$$ it follows that (\ref{h:control1}) holds.
\end{proof}

Thanks Propositions~\ref{h:prop:lfgn1}~and~\ref{h:lem:control1}, one knows that Theorem~\ref{h:thm:main1} is valid when $V_j$ is replaced by $\widetilde{V}_j$, for completing the proof of the theorem, one needs the following proposition, which shows that the asymptotic behaviors of $V_j$ and $\widetilde{V}_j$ are equivalent.

\begin{Prop}\label{h:prop:lfgn2}
One has, almost surely, 
\begin{equation}\label{h:lfgn2}
\frac{V_j}{\widetilde{V}_j} \xrightarrow[j\rightarrow+\infty]{a.s.} 1;
\end{equation}
recall that $V_j$ and $\widetilde{V}_j$ have been introduced, respectively in (\ref{h:def:VJ}) and (\ref{h:definition:Vjtilde:lfgn1}).
\end{Prop}

\begin{proof}[Proof of Proposition \ref{h:prop:lfgn2}]
Observe that, in view of (\ref{h:lfgn1}), it is sufficient to show that, one has almost surely, 
\begin{equation}\label{h:equi:lfgn2}
\frac{|V_j-\widetilde{V}_j|}{\ESP \big(\widetilde{V}_j \big)} \xrightarrow[j\rightarrow+\infty]{a.s.} 0.
\end{equation}
It follows from (\ref{h:def:VJ}), (\ref{h:definition:Vjtilde:lfgn1}), and inequality $\big||x|^{\beta}-|y|^{\beta}\big| \le |x-y|^{\beta}$ for any $(x,y)\in\R^2$, that for each $\o\in\Omega$ and $j\in\Z_+$,
$$
\big|V_j(\o)-\widetilde{V}_j(\o)\big|  \le n_j^{-1} \sum_{k\in\nu_j} \big|d_{j,k}(\o)-\widetilde{d}_{j,k}(\o)\big|^{\beta};
$$
then, (\ref{h:ecartdjktdjk}), entails that, 
\begin{equation}\label{h:maj1:lfgn2}
\big|V_j(\o)-\widetilde{V}_j(\o)\big| \le C_1(\o) 2^{-j\beta \rho_H},
\end{equation}
where $C_1$ is a finite random variable non depending on $j$. On the other hand, for every $j\ge j_0$, (\ref{eqa:esperancetildeV}), (\ref{eqa:equivmoS}), Lemma~\ref{h:lem:borne:scaletdjk}, (\ref{h:def:cardnuj}), (\ref{h:def:nuj}) and the inclusion $I_{j}\subseteq I_{j_0}$, imply that, 
\begin{equation}\label{h:maj2:lfgn2}
\ESP\big( \widetilde{V}_j \big) \ge c_{2} 2^{-j\beta \max_{t\in I_{j_0}} H(t) },
\end{equation}
where $j_0$ is as in the statement of Theorem~\ref{h:thm:main1}, and where $c_2>0$ is a constant non depending on $j$. Finally, putting together (\ref{h:maj1:lfgn2}), (\ref{h:maj2:lfgn2}) and (\ref{m:cond:pholder}), one obtains (\ref{h:equi:lfgn2}).
\end{proof}

Now, we are in position to end the prove of Theorem~\ref{h:thm:main1}.

\begin{proof}[End of the proof of Theorem \ref{h:thm:main1}] An easy computation allows to get, for all $j\in\Z_{+}$, that,
$$
\log_2(V_j)=\log_2\Big(\frac{V_j}{\widetilde{V}_j}\Big)+\log_2\Big(\frac{\widetilde{V}_j}{\ESP(\widetilde{V}_j)}\Big)+\log_2\big (\ESP(\widetilde{V}_j)\big);
$$
thus, it follows from Propositions~\ref{h:prop:lfgn2},~\ref{h:prop:lfgn1}~and~\ref{h:lem:control1}, that (\ref{h:main:equation}) holds.
\end{proof}

\section{Proof of Theorem~\ref{m:Th:main}}
\label{sec:proofm:Th:main}
The proof of Theorem~\ref{m:Th:main} relies on ideas, similar to the ones which have allowed to obtain Theorem~1.1 in \cite{hamonier2012lfsm}; recall that the latter theorem provides a wavelet estimator of the stability parameter $\al$ of Linear Fractional Stable Motion (LFSM), in the case where its Hurst parameter is known. Let us first present these ideas. Throughout, this section, for the sake of simplicity, one assumes that the interval $I=[0,1]$ (see the statement of Theorem~\ref{m:Th:main}), thus, 
for each integer $j\ge 2$, $D_j$ (see (\ref{eq:eqa:defDg})) can be expressed as, 
\begin{equation}
\label{m:eq:wavc-D0}
D_j=\max_{0\le k < 2^j} |d_{j,k}|;
\end{equation}
moreover, one sets 
\begin{equation}
\label{eqa:defHstar}
H_{*}=\min_{t\in[0,1]} H(t). 
\end{equation}
Then, observe that, in view of Part $3$ of Remark~\ref{Remh:thm:main1}, for proving Theorem~\ref{m:Th:main}, it is sufficient to show that $-j^{-1}\log_2 (D_j)$ provides a strongly
consistent estimator of $H_{*}-1/\al$; namely, one has almost surely,
\begin{equation}
\label{eqa:estim-lheX}
-j^{-1}\log_2 (D_j)\xrightarrow[j\rightarrow+\infty]{a.s.} H_* -1/\al;
\end{equation}
also, it is worth noticing that, thanks to (\ref{aeq:valYuhe}), one knows that $H_* -1/\al$ is in fact $\rho_{Y}^{\mbox{{\tiny unif}}}\big([0,1]\big)$, the uniform H\"older exponent over $[0,1]$, of the LMSM $\{Y(t)\,:\, t\in\R\}$. Next, observe that showing (\ref{eqa:estim-lheX}), is equivalent to show that the inequalities (\ref{m:eq:mainub}) and (\ref{m:minDj}) below, hold, for each fixed arbitrarily small real number $\epsilon>0$, almost surely:
\begin{equation}\label{m:eq:mainub}
\limsup_{j\rightarrow +\infty}\left\{ 2^{j(H_{*}-1/\al-\epsilon)}D_j \right\}<+\infty
\end{equation}
and 
\begin{equation}\label{m:minDj}
\liminf_{j\rightarrow +\infty} \left\{ 2^{j(H_{*}-1/\al+\epsilon)} D_j \right\}=+\infty.
\end{equation}
The inequality (\ref{m:eq:mainub}) follows from the fact that for all $\o\in\O$, one has, 
\begin{equation}
\label{m:eq:defC}
C_{1}(\omega) = \sup_{(t_1,t_2)\in [0,1]^2} \left\{ \frac{|Y(t_1,\omega)-Y(t_2,\omega)|}{|t_1-t_2|^{H_{*}-1/\al-\epsilon}}   \right\} <+\infty,
\end{equation}
which, in turn, is a straightforward consequence of the equality $\rho_{Y}^{\mbox{{\tiny unif}}}\big([0,1]\big)=H_* -1/\al$; namely, putting together, (\ref{m:djk}), (\ref{m:eq:support}), (\ref{m:2moments}) and (\ref{m:eq:defC}), one gets, for every $\o\in\O$, $j\ge 2$ and $k\in\{0,\ldots, 2^j-1\}$,
\begin{align*}
|d_{j,k}|& =2^j \Big|\int_{k2^{-j}}^{(k+1)2^{-j}} \left\{ Y(t,\o)-Y(k2^{-j},\o) \right\} \psi(2^jt-k) dt\Big|\\
|d_{j,k}(\omega)| & \le 2^j  \int_{k2^{-j}}^{(k+1)2^{-j}} \left|Y(t,\omega)-Y(k2^{-j},\omega)  \right| |\psi(2^jt-k)| dt \\
& \le C_{1}(\omega) \|\psi\|_{\Li(\R)}\, 2^j \int_{k2^{-j}}^{(k+1)2^{-j}} |t-k2^{-j}|^{H_{*}-1/\al-\epsilon}dt \\
&\le C_{1} (\omega) \|\psi\|_{\Li(\R)}\, 2^{-j(H_{*}-1/\al-\epsilon)}, 
\end{align*}
which, in view of (\ref{m:eq:wavc-D0}), means that (\ref{m:eq:mainub}) is true. 

From now on, we will focus on the proof of (\ref{m:minDj}). First, it is worth noticing that (\ref{Am:cond:pholder}) and the equality $I=[0,1]$, imply that there exist two positive constants $c$ and $\rho_H$, satisfying, 
\begin{equation}
\label{eqam:cond:holder}
\forall t_1,t_2\in [0,1], \, |H(t_1)-H(t_2)| \le c |t_1-t_2|^{\rho_H},
\end{equation}
and 
\begin{equation}
\label{eqam:cond:pholder}
1\ge \rho_H >\max_{t\in [0,1]} H(t).
\end{equation}
As, we have already pointed it out in the previous section, rather than directly working with wavelet coefficients $d_{j,k}$ (see (\ref{m:djk})), it is better to work with their appoximations $\widetilde{d}_{j,k}$ (see (\ref{h:tdjk})); so let us set, for all integer $j\ge 2$,
\begin{equation}\label{m:eq:wavc-D}
\widetilde{D}_j=\max_{0\le k < 2^j} |\widetilde{d}_{j,k}|.
\end{equation}
Putting together (\ref{m:djk}), (\ref{h:tdjk}), the change of variable $t=2^{-j}k+2^{-j}x$, (\ref{m:eq:support}), (\ref{eqa:uniflipX})\footnote{In which, one takes $\mathcal{I}=[0,1]$ and $\mathcal{H}=[\underline{H},\overline{H}]$.}, and (\ref{eqam:cond:holder}), one can show, similarly to (\ref{h:ecartdjktdjk}), that there exists a finite random variable $C_2$, non depending on $j$, such that for $\o\in\O$ and $j\ge 2$, one has,
$$
\max_{0\le k < 2^j} \big|d_{j,k}(\o)-\widetilde{d}_{j,k}(\o)\big|\le C_2(\o) 2^{-j\rho_H};
$$
thus using (\ref{eqam:cond:pholder}), (\ref{eqa:defHstar}), and the inequality,
$$
D_j(\o)\ge \widetilde{D}_j(\o)-\max_{0\le k < 2^j} \big|d_{j,k}(\o)-\widetilde{d}_{j,k}(\o)\big|,
$$
it turns out that for proving (\ref{m:minDj}), it is sufficient to show that, one has almost surely, for every $\epsilon>0$,
\begin{equation}
\label{m:mintDj}
\liminf_{j\rightarrow +\infty} \left\{ 2^{j(H_{*}-1/\al+\epsilon)} \widetilde{D}_j \right\}=+\infty.
\end{equation} 
It is worth noticing that (\ref{h:tdjk2}) and the inclusion in (\ref{h:local:phial}), imply that,  
\begin{equation}
\label{m:eq:decomp-djk}
\widetilde{d}_{j,k}=2^{-j(H(k2^{-j})-1/\al)} \int_{-\infty}^{(k+1)2^{-j}} \Phi_{\al}(2^js-k,H(k2^{-j})) \Za{ds}.
\end{equation}
Roughly speaking, the main idea for deriving (\ref{m:mintDj}), consists in expressing $\widetilde{d}_{j,k}$, as,
\begin{equation}
\label{eqa:decomptdjk}
\widetilde{d}_{j,k}=2^{-j(H(k2^{-j})-1/\al)}\big(g_{j,k}+r_{j,k}\big),
\end{equation}
where for each fixed integer $j\ge 2$, $\{g_{j,k}:0\le k <2^j\}$ is a finite sequence of independent $\sas$ random variables. Then, thanks to the nice properties of the latter sequence, using Borel-Cantelli Lemma, one can prove that, one has, almost surely,
\begin{equation}
\label{eqa:infgjk}
\liminf_{j\rightarrow +\infty} \left\{2^{j\epsilon}\max_{k\in \nu_j(t_0)}|g_{j,k}| \right\}=+\infty,
\end{equation} 
where the set of indices $\nu_j(t_0)$, defined in (\ref{m:def:nujt0}) below, is such that $2^{-jH(k2^{-j})}\asymp 2^{-jH_*}$.
On the other hand, Borel-Cantelli Lemma, allows to show that, the $\sas$ random variables $r_{j,k}$ satisfy, 
\begin{equation}
\label{eqa:suprjk}
\limsup_{j\rightarrow +\infty} \left\{2^{j\epsilon}\max_{k\in \nu_j(t_0)}|r_{j,k}| \right\}=0;
\end{equation}
which, in some sense, means that $\max_{k\in \nu_j(t_0)}|r_{j,k}|$ is negligible with respect to $\max_{k\in \nu_j(t_0)}|g_{j,k}|$. Finally, putting together, (\ref{eqa:decomptdjk}),
(\ref{eqa:infgjk}) and (\ref{eqa:suprjk}), one gets (\ref{m:mintDj}).

From now on, our goal will be to transform the latter heuristic proof of (\ref{m:mintDj}), in a rigorous proof; to this end, first, one needs to introduce some notations. 
\begin{itemize}
\item[$\bullet$] $t_0 \in [0,1]$ is a fixed point such that, 
\begin{equation}
\label{m:def:eq:mint0}
H(t_0)=H_{*}.
\end{equation}
\item[$\bullet$] For all integer $j\ge 2$, the compact interval $I_j(t_0)$, is defined as,  
\begin{equation}\label{m:def:eqIjt0}
I_j(t_0)=[0,1] \cap [t_0-j^{-\frac{1}{\rho_H}},t_0+j^{-\frac{1}{\rho_H}}],
\end{equation}
where $\rho_H$ has been introduced in (\ref{eqam:cond:holder});
notice that $|I_j(t_0)|$, the diameter of the latter interval, satisfies,
\begin{equation}\label{m:eq1:diamIjt0}
 j^{-\frac{1}{\rho_H}}\le |I_j(t_0)| \le 2 j^{-\frac{1}{\rho_H}}.
\end{equation}
\item[$\bullet$] $\nu_j(t_0)$ is the set of indices $k$, such that,
\begin{equation}\label{m:def:nujt0}
\nu_j(t_0)=\left\{ k\in\{0,1,\ldots,2^j-1\}\,:\, k2^{-j} \in I_j(t_0) \right\};
\end{equation}
observe that, in view of (\ref{m:eq1:diamIjt0}), one has for each $j$ big enough,
\begin{equation}\label{m:eq1:card:nujt0}
c_1 2^j j^{-\frac{1}{\rho_H}} \le \mathrm{card}(\nu_j(t_0)) \le c_2 2^j j^{-\frac{1}{\rho_H}},
\end{equation}
where $0<c_1\le c_2$ are two constants non depending on $j$.
\item[$\bullet$] The positive integer $e_j=e_j(\delta)$ is defined by,
\begin{equation}
\label{m:eq:add1}
e_j=\big[2^{j\delta}\big],
\end{equation}
where $\delta \in (0,1/3)$ is arbitrary and fixed.
\item[$\bullet$] $\widetilde{\nu}_j(t_0)$ denotes the subset of $\nu_j(t_0)$, such that,
\begin{equation}\label{m:def:equa:tildenujt0}
\widetilde{\nu}_j(t_0)=\big\{ k\in\nu_j(t_0)\, : \, \mbox{$k$ is divisible by $e_j$}\big\}=\big\{ k\in\nu_j(t_0)\, : \, \mbox{$\exists$ $l\in\Z_+$ \mbox{such that} $k=le_j$}\big\};
\end{equation}
observe that, in view of (\ref{m:eq1:card:nujt0}) and of (\ref{m:eq:add1}), one has for each $j$ big enough,
\begin{equation}\label{m:eq1:card:tildelambdajt0}
c_{1}' 2^{j(1-\delta)} j^{-\frac{1}{\rho_H}} \le \mathrm{card}(\widetilde{\nu}_j(t_0)) \le c_{2}' 2^{j(1-\delta)} j^{-\frac{1}{\rho_H}},
\end{equation}
where $0<c_{1}'\le c_{2}'$ are two constants non depending on $j$.
\item[$\bullet$] The set of indices $\widetilde{\lambda}_j(t_0)$ is defined as,
\begin{equation}\label{m:def:equa:tildelambdajt0}
\widetilde{\lambda}_j(t_0)=\big\{ l\in\Z_+\, : \, le_j \in \nu_j(t_0) \big\};
\end{equation}
observe that, one has,
\begin{equation}
\label{m:eq:egalitecard}
\mathrm{card}(\widetilde{\nu}_j(t_0))=\mathrm{card}(\widetilde{\lambda}_j(t_0)),
\end{equation}
since the natural map from $\widetilde{\lambda}_j(t_0)$ to $\widetilde{\nu}_j(t_0)$, namely the map $l\mapsto le_j$, is a bijection.
\item[$\bullet$] At last, for all integer $j$ large enough (so that $\widetilde{\lambda}_j(t_0)$ is non-empty), and for each $l\in\widetilde{\lambda}_j(t_0)$, one denotes by $G_{j,le_j}$ and $R_{j,le_j}$, the $\sas$ random variables, such that,
\begin{equation}\label{m:def:Gjl}
G_{j,le_j}= \int_{((l-1)e_j+1)2^{-j}}^{(le_j+1)2^{-j}} \Phi_{\al}(2^js-le_j,H(le_j2^{-j})) \Za{ds},
\end{equation}
and
\begin{equation}\label{m:def:Rjl}
R_{j,le_j}=\int_{-\infty}^{((l-1)e_j+1)2^{-j}} \Phi_{\al}(2^js-le_j,H(le_j2^{-j})) \Za{ds},
\end{equation}
then, it results from (\ref{m:eq:decomp-djk}) that,
\begin{equation}\label{m:rel:d:GR}
\widetilde{d}_{j,le_j}=2^{-j(H(le_j2^{-j})-1/\al)} \Big(G_{j,le_j} + R_{j,le_j} \Big);
\end{equation}
let us mention that, basically, $G_{j,le_j}$ and $R_{j,le_j}$ play the same roles as the $\sas$ random variables $g_{j,k}$ and $r_{j,k}$ (see (\ref{eqa:decomptdjk})) previously used in the heuristic proof of (\ref{m:mintDj}), yet, the index $k$ is now restricted to $k=le_j$; actually, such a restriction on $k$ is needed for technical reasons. 
\end{itemize}
Assuming, for a while, that the following two lemmas hold (for the sake of clarity, recall that the positive integer $e_j$ depends on $\delta$, more precisely it is given 
by (\ref{m:eq:add1})):

\begin{Lem}\label{m:prop:minmaxgjl}
One has, almost surely, for all $\delta\in (0,1/3)$,
\begin{equation}
\label{m:eq0:minmaxgjl}
\liminf_{j\rightarrow +\infty} \left\{2^{j\frac{2\delta}{\al}} \max_{l\in\widetilde{\lambda}_j(t_0)} |G_{j,le_j}|\right\} \ge 1.
\end{equation}
\end{Lem}

\begin{Lem}\label{m:prop:majmaxrjl}
One has, almost surely, for all $\delta\in (0,1/3)$,
\begin{equation}
\label{m:eq0:majmaxrjl}
\limsup_{j\rightarrow +\infty} \left\{ 2^{j\frac{2\delta}{\al}}\max_{l\in\widetilde{\lambda}_j(t_0)} |R_{j,le_j}| \right\}= 0.
\end{equation}
\end{Lem}
Then, (\ref{m:mintDj}) can be obtained as follows.
\begin{proof}[Proof of (\ref{m:mintDj})] Observe that (\ref{m:eq:wavc-D}) and the inclusion $\nu_j(t_0)\subseteq\{0,\ldots,2^{j}-1\}$, imply that,
$$
\widetilde{D}_j\ge \max_{k\in\nu_j(t_0)} |\widetilde{d}_{j,k}|;
$$
thus, for deriving (\ref{m:mintDj}), it is sufficient to show that, 
\begin{equation}
\label{m:min1tildedjk}
\liminf_{j\rightarrow +\infty} \left\{ 2^{j(H_{*}-1/\al+\epsilon)} \max_{k\in\nu_j(t_0)} |\widetilde{d}_{j,k}| \right\}= +\infty.
\end{equation}
Since $\epsilon$ is an arbitrarily small positive real number, there exists $\delta \in (0,1/3)$ such that,
\begin{equation}
\label{m:eq1:mainlb}
\epsilon/2=2\delta/\alpha.
\end{equation}
Then, using (\ref{m:def:equa:tildelambdajt0}), (\ref{m:eq1:mainlb}), (\ref{m:rel:d:GR}), (\ref{m:def:eq:mint0}), (\ref{m:def:nujt0}), (\ref{m:def:eqIjt0}), (\ref{eqam:cond:holder}), and the triangle inequality, one has, for any large enough integer $j$,
\begin{align*}
& 2^{j(H_{*}-1/\al+\epsilon/2)} \max_{k\in\nu_j(t_0)} |\widetilde{d}_{j,k}| \ge 2^{j(H_{*}-1/\al+\epsilon)} \max_{l\in \widetilde{\lambda}_j(t_0)} |\widetilde{d}_{j,le_j}|  \\
&\ge  2^{j\frac{2\delta}{\alpha}} \max_{l\in \widetilde{\lambda}_j(t_0)} \left| 2^{-j(H(le_j2^{-j})-H(t_0))} \left\{ G_{j,le_j}+R_{j,le_j} \right\} \right| \\
& \ge 2^{-c_1} 2^{j\frac{2\delta}{\alpha}} \max_{l\in \widetilde{\lambda}_j(t_0)} \left| G_{j,le_j}+R_{j,le_j} \right| \\
& \ge 2^{-c_1} 2^{j\frac{2\delta}{\alpha}} \max_{l\in \widetilde{\lambda}_j(t_0)} \left| G_{j,le_j}\right| - 2^{-c_1} 2^{j\frac{2\delta}{\alpha}} \max_{l\in \widetilde{\lambda}_j(t_0)} \left| R_{j,le_j}\right|,
\end{align*}
where $c_1$ is the constant $c$ in (\ref{eqam:cond:holder}); thus, (\ref{m:eq0:minmaxgjl}) and (\ref{m:eq0:majmaxrjl}) imply that, 
\begin{align*}
& \liminf_{j \rightarrow +\infty} \left\{ 2^{j(H_{*}-1/\al+\epsilon/2)} \max_{k\in\nu_j(t_0)} |\widetilde{d}_{j,k}| \right\} \\
& \ge 2^{-c_1} \liminf_{j\rightarrow +\infty} \left\{2^{j\frac{2\delta}{\alpha}} \max_{l\in \widetilde{\lambda}_j(t_0)} \left| G_{j,le_j}\right| \right\} - 2^{-c_1} \limsup_{j\rightarrow +\infty} \left\{2^{j\frac{2\delta}{\alpha}} \max_{l\in \widetilde{\lambda}_j(t_0)} \left| R_{j,le_j}\right| \right\}\ge 2^{-c_1}>0,
\end{align*}
which proves (\ref{m:min1tildedjk}).
\end{proof}

Now, our goal is to show that Lemma \ref{m:prop:minmaxgjl} holds, to this end one needs two preliminary results. The following lemma, whose proof can for instance be found
in \cite{SamTaq}, is a very classical result on the asymptotic behavior of the tail of a $\sas$ distribution\footnote{We note in passing, that Lemma~\ref{m:lem:tailZ} remains valid even if one drops the assumption of the symmetry of the stable distribution.}.
\begin{Lem}
\label{m:lem:tailZ}
 Let $S$ be a $\sas$ random variable with a non-vanishing scale parameter $\|S\|_{\al}$. Then for any real number $\xi\geq \|S\|_{\al}$, one has,
\begin{equation}\label{m:eq2:minmaxgjl}
c_{1} \|S\|_{\al}^{\al} \,\xi^{-\al} \leq \PR(|S|>\xi) \leq c_{2} \|S\|_{\al}^{\al}\, \xi^{-\al},
\end{equation}
where $0<c_{1}\le c_{2}$ are two constants only depending on $\al$.
\end{Lem}
\begin{Lem}\label{m:lem:majsclagjl}
For each arbitrary fixed integer $j$ large enough, $\{G_{j,le_j}\,:\, l\in\widetilde{\lambda}_j(t_0) \}$ is a sequence of independent $\sas$ random variables whose scale parameters are, for all $l\in\widetilde{\lambda}_j(t_0)$, given by 
\begin{equation}
\label{m:eq:scaleG}
\big\|G_{j,le_j} \big\|_{\al}^{\al} = 2^{-j} \int_{1-e_j}^1 \big|\Phi_{\al}(x,H(le_j2^{-j}))\big|^{\al} dx;
\end{equation}
moreover, one has,
\begin{equation}\label{m:eq:borne:scaleG}
c_{1}' 2^{-j} \le \big\|G_{j,le_j} \big\|_{\al}^{\al} \le c_{2}' 2^{-j}.
\end{equation}
where $0<c_{1}'\le c_{2}'$ are two constants non depending on $j$ and $l$.
\end{Lem}

\begin{proof}[Proof of Lemma \ref{m:lem:majsclagjl}]
First observe that the independence of the $\sas$ random variables $G_{j,le_j}$, $l\in\widetilde{\lambda}_j(t_0)$, is a consequence of the fact that they are defined by stable stochastic integrals (see (\ref{m:def:Gjl})) over the disjoint intervals $\big((l-1)e_j+1)2^{-j},(le_j+1)2^{-j}\big]$, $l\in\widetilde{\lambda}_j(t_0)$. Next, observe that, in view of (\ref{m:def:Gjl}) and of (\ref{eqa:scparamS}), $\big\|G_{j,le_j} \big\|_{\al}^{\al}$ can be expressed as, 
$$
\big\|G_{j,le_j} \big\|_{\al}^{\al}=\int_{((l-1)e_j+1)2^{-j}}^{(le_j+1)2^{-j}} \big|\Phi_{\al}(2^js-le_j,H(le_j2^{-j}))\big|^{\al} ds;
$$
thus, the change of variable $x=2^js-le_j$, allows to obtain (\ref{m:eq:scaleG}). Let us now show that (\ref{m:eq:borne:scaleG}) holds, to this end, one sets,
\begin{equation}
\label{eqa:defcts1}
c_2'=\max_{v\in [\underline{H},\overline{H}]}\int_{-\infty}^{1} \big|\Phi_{\al}(x,v)\big|^{\al} dx,\quad c_3'=\min_{v\in[\underline{H},\overline{H}]}\int_{-\infty}^{1} \big|\Phi_{\al}(x,v)\big|^{\al} dx,\quad\mbox{and}\quad c_1' =2^{-1}c_3';
\end{equation}
one recalls, in passing, that the range of the function $H(\cdot)$ is included in the interval $[\underline{H},\overline{H}]$. Notice that, it follows from (\ref{eqa:defcts1}) as well as from Parts $(ii)$ and $(iii)$ of Proposition~\ref{h:prop:proprietephial}, that $0<c_1'<c_3'\le c_2' <+\infty$. Next, combining (\ref{m:eq:scaleG}) with the first equality in (\ref{eqa:defcts1}), one can easily gets the second inequality in (\ref{m:eq:borne:scaleG}). On the other hand, (\ref{m:eq:scaleG}), the second equality in (\ref{eqa:defcts1}), and (\ref{h:local:phial}), imply that,
\begin{align}
\label{eqa:minor1}
\big\|G_{j,le_j} \big \|_{\al}^{\al} & \ge 2^{-j}\left(c_3'-\int_{-\infty}^{1-e_j} \big|\Phi_{\al}(x,H(le_j2^{-j}))\big|^{\al} dx \right)\nonumber\\
& \ge 2^{-j}\left( c_{3}' -c_{4}'\int_{-\infty}^{1-e_j} \big(1-x\big)^{-(2\al+1-\al\overline{H})}\right)\nonumber\\
& \ge 2^{-j}\left( c_{3}' -c_{5}' e_{j}^{-\al(2-\overline{H})}\right),
\end{align}
where $c_4'>0$ and $c_5'>0$ are two constants non depending on $j$. Finally, putting together (\ref{eqa:minor1}), the third equality in (\ref{eqa:defcts1}), and the fact that $e_j$ is a non-decreasing sequence which goes to $+\infty$, one obtains the first inequality in (\ref{m:eq:borne:scaleG}).
\end{proof}

We are now in position to prove Lemma \ref{m:prop:minmaxgjl}.

\begin{proof}[Proof of Lemma \ref{m:prop:minmaxgjl}]
Let $j$ be an arbitrary integer, large enough, so that the set $\lambda_j(t_0)$ is non-empty and the inequality, 
\begin{equation}
\label{eqa:maxscaleG}
2^{-j \frac{2\delta}{\al}}\ge\max_{l\in\widetilde{\lambda}_j(t_0) } \|G_{j,le_j}\|_{\al},
\end{equation}
holds; notice that in view of the second inequality in (\ref{m:eq:borne:scaleG}), and of the fact that $\delta\in (0,1/3)$, one can assume that (\ref{eqa:maxscaleG}) is satisfied.
Next, using the fact that $\{G_{j,le_j}\,;\, l\in\widetilde{\lambda}_j(t_0) \}$ is a finite sequence of independent random variables, one gets that,
\begin{align*}
& \PR\left(\max_{l\in\widetilde{\lambda}_j(t_0) } |G_{j,le_j}| \le 2^{-j \frac{2\delta}{\al} } \right) = \prod_{l\in\widetilde{\lambda}_j(t_0)} \PR\left(|G_{j,le_j}| \le 2^{-j \frac{2\delta}{\al} } \right)  \\
& = \prod_{l\in\widetilde{\lambda}_j(t_0)} \Bigg( 1 - \PR\left(|G_{j,le_j}| > 2^{-j \frac{2\delta}{\al} } \right) \Bigg). 
\end{align*}
Moreover, thanks to (\ref{eqa:maxscaleG}), one is allowed to apply the first inequality in (\ref{m:eq2:minmaxgjl}), in the case where $\xi=2^{-j \frac{2\delta}{\al}}$ and $S=G_{j,le_j}$, $l\in\widetilde{\lambda}_j(t_0)$ being arbitrary; thus, one obtains that,
\begin{align}
\label{m:eq3:minmaxgjl}
\PR\left(\max_{l\in\widetilde{\lambda}_j(t_0) } |G_{j,le_j}| \le 2^{-j \frac{2\delta}{\al} } \right) & \le \prod_{l\in\widetilde{\lambda}_j(t_0)} \left( 1 - c_1\|G_{j,le_j}\|_{\al}^{\al}\, 2^{j 2\delta}\right)\nonumber\\
& \le \left(1-  c_{2} 2^{-j(1-2\delta)}\right)^{\mathrm{card}(\widetilde{\lambda}_j(t_0))},
\end{align}
where, $c_1>0$ is the same constant as in (\ref{m:eq2:minmaxgjl}), and $c_2\in (0,1)$ is a constant non depending on $j$; observe that the last inequality in (\ref{m:eq3:minmaxgjl}), results from the first inequality in (\ref{m:eq:borne:scaleG}).

Finally, putting together (\ref{m:eq3:minmaxgjl}), (\ref{m:eq1:card:tildelambdajt0}) and (\ref{m:eq:egalitecard}), it follows that,
$$
\sum_{j=j_1}^{+\infty} \,\PR\left(\max_{l\in\widetilde{\lambda}_j(t_0) } |G_{j,le_j}| \le 2^{-j \frac{2\delta}{\al} } \right) <+\infty,
$$
where $j_1$ denotes a fixed large enough integer; then, applying Borel-Cantelli Lemma, one gets (\ref{m:eq0:minmaxgjl}).
\end{proof}

In order to prove Lemma~\ref{m:prop:majmaxrjl}, one needs the following result. 

\begin{Lem}\label{m:lem:majsclarjl}
For each integer $j$ large enough, and for every $l\in \widetilde{\lambda}_j(t_0)$, the scale parameter $\|R_{j,le_j} \|_{\al}$ of $R_{j,le_j}$, the $\sas$ random variable introduced in (\ref{m:def:Rjl}), satisfies
\begin{equation}
\label{m:eq:scaleR}
\|R_{j,le_j} \|_{\al}^{\al} = 2^{-j} \int_{-\infty}^{1-e_j} \big|\Phi_{\al}(x,H(le_j2^{-j}))\big|^{\al} dx\leq c \,2^{-j\al(2\delta +1/\al-\delta \overline{H})},
\end{equation}
where $c>0$ is a constant non depend on $j$ and $l$.
\end{Lem}

\begin{proof}[Proof of Lemma \ref{m:lem:majsclarjl}]
First observe that in view of (\ref{m:def:Rjl}) and of (\ref{eqa:scparamS}), $\|R_{j,le_j} \|_{\al}^{\al}$ can be expressed as,
\begin{equation*}
\|R_{j,le_j} \|_{\al}^{\al} = \int_{-\infty}^{((l-1)e_j+1)2^{-j}}\big|\Phi_{\al}(2^js-le_j,H(le_j2^{-j}))\big|^{\al} ds,
\end{equation*}
thus, the change of variable $u=2^js-le_j$, allows to obtain the equality in (\ref{m:eq:scaleR}). Then, using the latter equality, (\ref{h:local:phial}) and (\ref{m:eq:add1}), it follows that the inequality in (\ref{m:eq:scaleR}) holds, more precisely, one has,
$$
\|R_{j,le_j} \|_{\al}^{\al}  \le c_1 2^{-j} \int_{-\infty}^{1-e_j} (1-u)^{-(2\al+1-\al\overline{H})} du \le c_1  2^{-j} \int_{2^{j\delta}-2}^{+\infty} (1+u)^{-(2\al+1-\al \overline{H})} du \le c_{2} 2^{-j\al(2\delta+1/\al-\delta\overline{H})},
$$
where $c_1>0$ and $c_2>0$ are two constants non depending on $j$.
\end{proof}

We are now in position to prove Lemma \ref{m:prop:majmaxrjl}.

\begin{proof}[Proof of Lemma \ref{m:prop:majmaxrjl}]
First observe that for all fixed arbitrarily small $\eta>0$, one has,  
\begin{equation}
\label{eqa:inegd1}
\frac{2\delta+\eta}{\al}<2\delta+1/\al-\delta \overline{H},
\end{equation}
since $\delta \in (0,1/3)$, $\al\in (1,2)$ and  $\overline{H}\in (0,1)$. Next, combining (\ref {m:eq:scaleR}) with (\ref{eqa:inegd1}), it follows that the inequality,
\begin{equation}
\label{eqa:inegd2}
2^{-j\big (\frac{2\delta+\eta}{\al}\big)}\ge \max_{l\in\widetilde{\lambda}_j(t_0) }\|R_{j,le_j} \|_{\al},
\end{equation}
holds for all $j$ big enough. Thanks to (\ref{eqa:inegd2}), one is allowed to apply the second inequality in (\ref{m:eq2:minmaxgjl}), in the case where $\xi=2^{-j\big (\frac{2\delta+\eta}{\al}\big)}$ and $S=R_{j,le_j}$, $l\in\widetilde{\lambda}_j(t_0)$ being arbitrary; thus, one obtains that,
\begin{align}
\label{m:eq1:majmaxrjl}
& \PR\Bigg(\max_{l\in\widetilde{\lambda}_j(t_0) } |R_{j,le_j}| > 2^{-j\big(\frac{2\delta+\eta}{\al}\big)} \Bigg) \le \sum_{l\in\widetilde{\lambda}_j(t_0)} \PR\Bigg(|R_{j,le_j}| > 2^{-j\big(\frac{2\delta+\eta}{\al}\big)} \Bigg) \nonumber \\
& \le c_{2} \sum_{l\in\widetilde{\lambda}_j(t_0)} \|R_{j,le_j} \|_{\al}^{\al}\, 2^{j(2\delta+\eta)},
\end{align}
where $c_{2}>0$ is the same constant as (\ref{m:eq2:minmaxgjl}). Next, putting together (\ref{m:eq1:majmaxrjl}), the inequality in (\ref{m:eq:scaleR}), (\ref{m:eq:egalitecard}) and (\ref{m:eq1:card:tildelambdajt0}), one gets that,
\begin{equation}\label{m:eq2:majmaxrjl}
\PR\Bigg(\max_{l\in\widetilde{\lambda}_j(t_0) } |R_{j,le_j}| > 2^{-j\big(\frac{2\delta+\eta}{\al}\big)} \Bigg) \le c_3 j^{-\frac{1}{\rho_H}} 2^{-j\al(2\delta+1/\al-\delta \overline{H})}\, 2^{j(1+\delta+\eta)},
\end{equation}
where $c_3>0$ is a constant non depending on $j$. Moreover, the fact $\eta$ is arbitrarily small, allows to assume that $\delta(\al-1) > \eta$; thus, one obtains 
\begin{equation}
\label{eqa:inegd3}
\al(2\delta+1/\al-\delta\overline{H}) > \al(\delta+1/\al)=\al\delta+1 > 1+\delta +\eta.
\end{equation}
Finally, combining (\ref{m:eq2:majmaxrjl}) with (\ref{eqa:inegd3}), it follows that, 
\begin{equation*}
\sum_{j=j_1}^{+\infty} \PR\Bigg(\max_{l\in\widetilde{\lambda}_j(t_0) } |R_{j,le_j}| > 2^{-j\big(\frac{2\delta+\eta}{\al}\big)} \Bigg) < +\infty,
\end{equation*}
where $j_1$ denotes a fixed large enough integer; then, applying Borel-Cantelli Lemma, one gets (\ref{m:eq0:majmaxrjl}).
\end{proof}

\begin{Rem}\label{m:rem:unif-event}
Our proofs of Lemmas~\ref{m:prop:minmaxgjl}~and~\ref{m:prop:majmaxrjl}, only allow to derive that Relations (\ref{m:eq0:minmaxgjl})~and~(\ref{m:eq0:majmaxrjl}) hold on 
some event of probability $1$, denoted by $\widetilde{\Omega}_\delta$, since it a priori depends on $\delta\in (0,1/3)$. Yet, one can easily show that these two relations
also hold, for every real number $\delta\in (0,1/3)$, on an event of probability $1$ which does not depend on $\delta$, namely the event $\bigcap_{\delta\in\Q\cap
(0,1/3)} \widetilde{\Omega}_\delta$.
\end{Rem}
\vspace{0.5cm}
\noindent{\bf Aknowledgment.} This work has been partially supported by CEMPI (ANR-11-LABX-0007-01). 	
\thispagestyle{plain}
\bibliographystyle{amsplain}
\bibliography{mabibliodethese}
\end{document}